\documentclass{amsart}

\usepackage{amssymb}
\usepackage{latexsym}
\usepackage{amssymb,amsmath}
\usepackage{tikz}
\usepackage{mathdots}
\usepackage{graphics}
\usepackage{url}
\usepackage[vcentermath]{youngtab}  
\usepackage{ctable}
\newcommand{\syoung}[1]{\scalebox{0.75}{$\displaystyle\young#1$}}
\usepackage{amsthm}

\newtheorem{theorem}{Theorem}[section]
\newtheorem{lemma}[theorem]{Lemma}
\newtheorem{proposition}[theorem]{Proposition}
\newtheorem{corollary}[theorem]{Corollary}

\theoremstyle{definition}
\newtheorem{definition}[theorem]{Definition}
\newtheorem{example}[theorem]{Example}

\theoremstyle{remark}
\newtheorem{remark}[theorem]{Remark}

\numberwithin{equation}{section}

\newcommand{\N}{\mathbf{N}}
\newcommand{\Z}{\mathbf{Z}}

\newcommand{\C}{\mathbf{C}}

\renewcommand{\epsilon}{\varepsilon}

\DeclareMathOperator{\GL}{GL}

\DeclareMathOperator{\Sym}{Sym}

\DeclareMathOperator{\sgn}{sgn}
\DeclareMathOperator{\Inf}{Inf}
\DeclareMathOperator{\SSYT}{SSYT}

\DeclareMathOperator{\mon}{mon}
\DeclareMathOperator{\wt}{wt}

\newcommand{\medcup}{\scalebox{0.85}{$\,\bigcup\,$}}

\newcommand{\Ind}{\big\uparrow}

\newcommand{\ind}{\!\uparrow}

\renewcommand{\theta}{\vartheta}
\newcommand{\id}{\mathrm{id}}

\newcommand{\GLEmodr}[1]{\GLEmod_#1}
\newcommand{\GLEmod}{\GL(E)\text{\raisebox{0.5pt}{-}mod}}
\newcommand{\gl}{\mathsf{gl}}
\renewcommand{\b}{\mathsf{b}}
\newcommand{\CPP}{\mathrm{CPP}}	

\newcommand{\Par}{\mathrm{Par}}

\newcommand{\rr}{\ell}

\newcommand{\doubletableau}[4]{%
\begin{tikzpicture}[x=#1,y=-#2]
\draw[thick](0,0)--(2,0);\draw[thick](0,4)--(2,4);
\draw[thick](0,0)--(0,4);\draw[thick](1,0)--(1,4);
\draw[thick](2,0)--(2,4);\node at (0.5,2) {#3}; 
\node at (1.5,2) {#4};
\end{tikzpicture}}
\newcommand{\ltableau}[5]{%
\begin{tikzpicture}[x=#1,y=-#2]
\draw[thick](0,0)--(1.8,0);\draw[thick](0,4)--(1.8,4);
\draw[thick](0,0)--(0,8);\draw[thick](1,0)--(1,8);
\draw[thick](1.8,0)--(1.8,4);\node at (0.5,2) {#3}; 
\draw[thick](0,8)--(1,8);
\node at (1.4,2) {#4};\node at (0.5,6) {#5};
\end{tikzpicture}}
\newcommand{\thindoubletableau}[4]{%
\begin{tikzpicture}[x=#1,y=-#2,line width=0.65pt]
\draw(0,0)--(2,0);\draw(0,4)--(2,4);
\draw(0,0)--(0,4);\draw(1,0)--(1,4);
\draw(2,0)--(2,4);\node at (0.5,2) {#3}; 
\node at (1.5,2) {#4};
\end{tikzpicture}}
\newcommand{\triplevtableau}[3]{
\begin{tikzpicture}[x=0.4cm,y=-0.6cm,line width=1pt]
\draw(0,0)--(2,0)--(2,9)--(0,9)--(0,0);\draw (0,3)--(2,3);\draw (0,6)--(2,6);
\node at (1,1.5) {#1};
\node at (1,4.5) {#2};
\node at (1,7.5) {#3};
\end{tikzpicture}}
\newcommand{\tripletableau}[3]{
\begin{tikzpicture}[x=0.6cm,y=-0.4cm,line width=0.8pt]
\draw(0,0)--(9,0)--(9,2)--(0,2)--(0,0);\draw (3,0)--(3,2);\draw (6,0)--(6,2);
\node at (1.5,1) {#1};
\node at (4.5,1) {#2};
\node at (7.5,1) {#3};
\end{tikzpicture}}

\newcommand{\cp}{\hskip-1.25pt+\hskip-1.25pt}	

\newcounter{thmlistcnt}
\newenvironment{thmlist}%
	{\setcounter{thmlistcnt}{0}%
	\begin{list}{\emph{(\roman{thmlistcnt})}}{%
		\usecounter{thmlistcnt}%
		\setlength{\topsep}{0pt}%
		\setlength{\leftmargin}{0pt}%
		\setlength{\itemsep}{0pt}%
		\setlength{\labelwidth}{17pt}
		\setlength{\itemindent}{30pt}}%
	}%
	{\end{list}}%
	
\newcounter{defnlistcnt}
\newenvironment{defnlist}%
	{\setcounter{defnlistcnt}{0}%
	\begin{list}{(\roman{defnlistcnt})}{%
		\usecounter{defnlistcnt}%
		\setlength{\topsep}{0pt}%
		\setlength{\leftmargin}{32pt}%
		\setlength{\itemsep}{0pt}%
		\setlength{\labelwidth}{32pt}
		\setlength{\itemindent}{0pt}}%
	}%
	{\end{list}}%
	
\newcommand{\SymG}{S}
\newcommand{\df}{\delta} 
\newcommand{\s}{\ell} 
\newcommand{\V}{\hskip-1pt V}
\newcommand{\B}{\mathcal{B}}
\newcommand{\nablanumuE}{\nabla^\nu\bigl(\nabla^\mu(E)\bigr)}	
\newcommand{\nablanumuEt}{\nabla^\nu\bigl(\nabla^{(r) \sqcup \mu}(E^+)\bigr)}	
\newcommand{\nablanumuEh}{\nabla^\nu\bigl(\nabla^{\mu + (1^r)}(E)\bigr)}	
\newcommand{\Ftt}{F(\hskip0.5pt\widetilde{\hskip0.25pt t}\hskip1pt)}
\newcommand{\Fttphic}{F(\hskip0.5pt\widetilde{\hskip0.25pt t}\hskip1pt \phi_c)}
\newcommand{\altbox}[1]{[#1]}
\newcommand{\Bp}{{\mathcal{B}^+}}   
\newcommand{\CC}{{\mathcal{C}^\star}}     
\newcommand{\BB}{{{\mathcal{B}^+}}}
\newcommand{\BBN}[1]{{\mathcal{B}^{+#1}}}     
\newcommand{\e}{e}                

\newcommand{\pairpq}{(I,J)}
\renewcommand{\u}{e}                

\newcommand{\bt}[1]{#1^\star}

\newcommand{\nstar}{n^{\raisebox{-1pt}{$\hskip-0.5pt\scriptscriptstyle \star$}}}

\begin{document}

\title[Plethysms of symmetric functions]{Plethysms of symmetric functions and
highest weight representations}


\author{Melanie de Boeck}
\address{School of Mathematics, Statistics and Actuarial Science,    University of Kent,   Canterbury,    CT2 7FS, U.K.}
\curraddr{}
\email{melaniedeboeck@hotmail.com}
\thanks{}

\author{Rowena Paget}
\address{School of Mathematics, Statistics and Actuarial Science,   University of Kent,   Canterbury,    CT2 7FS, U.K.}
\curraddr{}
\email{r.e.paget@kent.ac.uk}
\thanks{}

\author{Mark Wildon}
\address{Department of Mathematics,  Royal Holloway,  University of London,     Egham, TW20 0EX, U.K.}
\curraddr{}
\email{mark.wildon@rhul.ac.uk}
\thanks{}

\date{}

\dedicatory{}

\begin{abstract}
Let $s_\nu \circ s_\mu$ denote the plethystic product of the
Schur functions $s_\nu$ and $s_\mu$. In this article we
define an explicit
polynomial representation corresponding to $s_\nu \circ s_\mu$ with basis 
indexed by certain `plethystic' semistandard tableaux. Using these representations
we prove generalizations of four
results on plethysms due to Bruns--Conca--Varbaro, Brion, Ikenmeyer and
the authors.
In particular, we give a sufficient condition for the
multiplicity $\langle s_\nu \circ s_\mu, s_\lambda\rangle$
to be stable under insertion of new parts into $\mu$ and $\lambda$.
We also
characterize all
maximal and minimal partitions $\lambda$ in the dominance order such that $s_\lambda$
appears in $s_\nu \circ s_\mu$ and determine the corresponding multiplicities
using plethystic semistandard tableaux.
\end{abstract}

\maketitle

\section{Introduction}
\thispagestyle{empty}

Let $s_\lambda$ denote the Schur function labelled by the partition $\lambda$.
Expressing a general plethysm $s_\nu \circ s_\mu$ as a non-negative linear combination
of Schur functions has been identified by Stanley as a fundamental open problem in
symmetric combinatorics \cite[Problem 9]{StanleyPositivity}.
While many partial results are known, often obtained by  deep 
combinatorial, algebraic or geometric arguments,
a satisfying general solution remains out of reach.

In this article we generalize four results on plethysms due to 
Bruns--Conca--Varbaro \cite{BCV}, Brion \cite{Brion}, Ikenmeyer \cite{IkenmeyerThesis}, 
and the authors \cite{deBoeck, PagetWildonTwisted}, and give them unified proofs in the 
setting of polynomial representations of the general linear group. Our proofs are essentially
elementary, requiring little more than basic multilinear algebra and the background
recalled in \S\ref{sec:background}. 
The article is intended to be readable by non-experts:
in particular, we define the Schur functions $s_\lambda$ and the plethysm product~$\circ$
in~\S\ref{subsec:plethysm} and include several examples in this introduction.

To state our main results we need the following notation.
Let $\Par(r)$ denote the set of partitions of $r \in \N_0$ and let $\ell(\lambda)$
denote the number of parts of the partition~$\lambda$. 
Given partitions
$\lambda$ and $\mu$, let $\lambda \sqcup \mu$
be the partition whose multiset of parts is the disjoint union of the multisets of parts of $\lambda$
and $\mu$. Define $\lambda + \mu$  by $(\lambda+\mu)_i = \lambda_i + \mu_i$.
(As a standing convention, if $i > \ell(\lambda)$ then we
set $\lambda_i = 0$.)
For $N \in \N_0$ define $N\lambda$  by $(N\lambda)_i = N\lambda_i$ for each~$i$.
Let $\lambda'$ denote the conjugate of the partition $\lambda$.

Fix $\mu \in \Par(m)$, $\nu \in \Par(n)$ and $\lambda \in \Par(mn)$.

%


\begin{theorem}\label{thm:BCV-N}
If $r$ is at least the greatest part of $\mu$ then
\[ \langle s_\nu \circ s_{(r) \sqcup \mu},
s_{(nr) \sqcup \lambda} \rangle = \langle s_\nu \circ s_\mu , s_\lambda \rangle. \]
\end{theorem}



\begin{theorem}\label{thm:Brion}
If $r \in \N$ then
\[ \langle s_\nu \circ s_{\mu +  (1^r)},
s_{\lambda + (n^r)} \rangle \ge \langle s_\nu \circ s_\mu, s_\lambda \rangle. \] 
Moreover $\langle s_\nu \circ s_{\mu + N (1^r)}, s_{\lambda + N(n^r)} \rangle$
is constant for $N \in \N_0$ such that
\[ N \ge n(\mu_1 + \cdots + \mu_{r-1}) + (n-1)\mu_r + \mu_{r+1} -
(\lambda_1 + \cdots + \lambda_r). \]
\end{theorem}

In particular, $\langle s_\nu \circ s_{\mu +  (1^r)}, s_{\lambda + (n^r)} \rangle = \langle s_\nu \circ s_\mu, s_\lambda \rangle$
if $r \ge \ell(\mu)$ and $r \ge \ell(\lambda)$.
We give an upper bound for the stable multiplicity at the end of \S\ref{sec:Brion}.



\begin{theorem}\label{thm:Ikenmeyer}
If $n^\star \in \N$, $\lambda^\star \in \Par(mn^\star)$ and
$\langle s_{(n^\star)} \circ s_\mu, s_{\lambda^\star}\rangle \ge 1$ then
\[ \langle
s_{(n + n^\star)} \circ s_\mu, s_{\lambda + \lambda^\star} \rangle
\ge  \langle s_{(n)} \circ s_\mu, s_\lambda \rangle
. \]
\end{theorem}


For our final theorem we need some further combinatorial definitions.
The dominance order
on partitions is the partial order
defined  by $\lambda \unrhd \mu$ if $\lambda_1 + \cdots + \lambda_i \ge \mu_1 + \cdots + \mu_i$
for all $i \in \N$.
Semistandard tableaux, with entries from an arbitrary totally ordered set, 
are defined in \S\ref{subsec:pts} below. In particular, $\SSYT_\N(\mu)$ denotes
the set of semistandard $\mu$-tableaux with entries from $\N$.
We order $\SSYT_\N(\mu)$ by
the total order defined in Definition~\ref{defn:totalorder}.

\begin{definition}{\ }\label{defn:pssyt}
\begin{defnlist}


\item A \emph{plethystic semistandard tableau} of \emph{shape} $\mu^\nu$ is
a semistandard $\nu$-tableau whose entries are tableaux in $\SSYT_\N(\mu)$.

\item Let $T$ be a plethystic semistandard tableau and let $M$ be the greatest 
entry of the tableau entries of $T$.
The \emph{weight} of $T$,  denoted $\wt(T)$, is 
the composition $(\beta_1,\ldots,\beta_M)$
such that for each $b$, the total number of occurrences of~$b$
in the tableau entries of $T$ is $\beta_b$.
\end{defnlist}
\end{definition}

These objects are illustrated in Example~\ref{ex:ex}.

\begin{theorem}\label{thm:maxls}
The maximal
partitions $\lambda$ in the dominance order
such that $s_\lambda$ is a constituent of
$s_\nu \circ s_\mu$ are precisely the maximal weights
of the plethystic semistandard tableaux of shape $\mu^{\nu}$.
Moreover if $\lambda$ is such a maximal partition then
$\langle s_\nu \circ s_\mu, s_\lambda \rangle$ is the number of plethystic
semistandard tableaux of shape~$\mu^\nu$ and weight $\lambda$.
\end{theorem}



Applying the sign twist in Lemma~\ref{lemma:signtwist} to the main theorems
gives equivalent results that are also noteworthy. In particular, Theorem~\ref{thm:BCV-N} 
implies that if $r \ge \ell(\mu)$ 
then
\begin{equation}\label{eq:BCV-Ntwisted}
\langle s_{\kappa} \circ s_{\mu+(1^r)},
s_{\lambda + (1^{nr})} \rangle
= \langle s_\nu \circ s_\mu, s_\lambda \rangle,
\end{equation}
where $\kappa = \nu$ if $r$ is even and $\kappa = \nu'$ if $r$ is odd.
The sign-twist of Theorem~\ref{thm:maxls} characterizes the minimal partitions $\lambda$
such that $s_\lambda$ appears in a general plethysm $s_\nu \circ s_\mu$.

In the setting of polynomial representations of general linear groups,
the Schur function $s_\lambda$ corresponds to the Schur functor $\nabla^\lambda$. We develop this background
in~\S\ref{sec:background}.
In~\S\ref{sec:model} we construct an explicit model for the module $\nabla^\nu \bigl(\nabla^\mu (E) \bigr)$, where~$E$
is a complex vector space. By Proposition~\ref{prop:compositionPlethysm}, the formal character of this module is
$(s_\nu \circ s_\mu)(x_1,\ldots, x_d)$, where $d = \dim E$. Using this model we prove
Theorems~
\ref{thm:BCV-N},~\ref{thm:Brion},~\ref{thm:Ikenmeyer},~\ref{thm:maxls} 
in \S\ref{sec:BCV-N}, \S\ref{sec:Brion}, \S\ref{sec:Ikenmeyer} and~\S\ref{sec:maxls}, respectively.


\begin{example}\label{ex:ex}
As an illustration of our  main theorems, 
we determine the plethysm $s_{(3)} \circ s_{(3)}$. By Lemma~\ref{lemma:signtwist},
its sign twist is
 $s_{(1^3)} \circ s_{(1^3)}$. Using the closure condition in Definition~\ref{defn:closureAlt}
 it is easy to show that there are two plethystic semistandard
tableaux of shape $(1^3)^{(1^3)}$ and maximal weight,
as shown in margin. The weights are $(3,3,1,1,1)$ and $(3,2,2,2)$ respectively.%
\begin{marginpar}{\vskip-48pt\scalebox{0.9}{\triplevtableau{\young(1,2,3)}{\young(1,2,4)}{\young(1,2,5)}}\raisebox{1in}{\,,} \ 
\scalebox{0.9}{\triplevtableau{\young(1,2,3)}{\young(1,2,4)}{\young(1,3,4)}}} \end{marginpar}
Hence, by Theorem~\ref{thm:maxls} and Lemma~\ref{lemma:signtwist},  
\[ \smash{\langle s_{(3)} \circ s_{(3)}, s_{(5,2,2)} \rangle
= \langle s_{(3)} \circ s_{(3)}, s_{(4,4,1)} \rangle = 1}. \] 
Since $s_{(1^3)} \circ s_{\varnothing} = s_{\varnothing}$, it follows from~\eqref{eq:BCV-Ntwisted},
applied with $r=3$, that
$\langle s_{(1^3)} \circ s_{(1^3)}, s_{(1^9)}\rangle = 1$, and so 
$\langle s_{(3)} \circ s_{(3)}, s_{(9)} \rangle = 1$. (This also follows from Theorem~\ref{thm:maxls}, since
the unique plethystic semistandard tableau of shape $(3)^{(3)}$ and maximal
weight~is
\[ 
\tripletableau{\young(111)}{\young(111)}{\young(111)}
\]
and can, of course, be seen in many other ways.)

By Example~\ref{ex:hwFoulkes}, $s_{(3)} \circ s_{(2)} = s_{(6)} + s_{(4,2)} + s_{(2,2,2)}$.
By the final part of Theorem~\ref{thm:Brion}, applied to the second summand, 
$\langle s_{(3)} \circ s_{(2+N)}, s_{(4+3N,2)} \rangle$ is constant for $N \ge 0$,
hence $\langle s_{(3)} \circ s_{(3)}, s_{(7,2)} \rangle = 0$.
Moreover, by the Cayley--Sylvester formula (see for instance \cite[Exercise 6.18, solution]{FultonHarrisReps}), 
$\langle s_{(3)} \circ s_{(3)}, s_{(6,3)}\rangle$ is the number of partitions
of $3$ contained in the $3 \times 3$ box, minus the number of partitions of $2$ satisfying
the same restriction. Therefore $\langle s_{(3)} \circ s_{(3)}, s_{(6,3)}\rangle = 1$.
Hence 
\begin{equation}
\label{eq:s33} 
s_{(3)} \circ s_{(3)} = s_{(9)} + s_{(7,2)} + s_{(6,3)} + s_{(5,2,2)} + s_{(4,4,1)} + f 
\end{equation}
for some symmetric function $f$ with non-negative coefficients in the Schur basis. Under the
characteristic isometry (see \S\ref{subsec:plethysm}), $s_{(3)} \circ s_{(3)}$ is the image
of the permutation character of $\SymG_9$ acting on the set partitions of $\{1,\ldots, 9\}$ into
$3$ sets, each of size~$3$. This character has degree $9!/3!^33! = 280$, which equals the sum 
of the degrees of the irreducible characters of $\SymG_9$ 
labelled by the partitions appearing in~\eqref{eq:s33}. 
Hence $f=0$.


\end{example}

We now explain the antecedents of the main theorems, before
giving a more detailed outline of the strategy of our proofs.
For general background on symmetric functions, including the Young and Pieri rules,
we refer the reader to \cite[Ch.~7]{StanleyII} or \cite[Ch.~1]{MacDonald}. 
For more background on plethysms
see \cite{LoehrRemmel} and the introduction~to~\cite{PagetWildonGeneralizedFoulkes}. 

\bigskip

\begin{samepage}
\subsection*{Antecedents of the main theorems}

%

\subsubsection*{Theorem~\ref{thm:BCV-N}}
Stated in the language of symmetric functions, 
Proposition~1.16 of \cite{BCV} becomes
\begin{equation}\label{eq:BCV}\langle s_{\nu} 
\circ s_{(1^{m+1})}, s_{(n) \sqcup \lambda} \rangle
=  \langle s_{\nu} \circ s_{(1^m)}, s_{\lambda} \rangle 
, \end{equation}
provided that $n$ is at least the first part of $\lambda$.
By Remark~\ref{remark:BCVnz}, both sides in Theorem~\ref{thm:BCV-N} are zero
if $\lambda_1 > nr$. Therefore this proposition is equivalent to the case
$\mu = (1^m)$ and $r=1$ of Theorem~\ref{thm:BCV-N}. 
(The statement in \cite{BCV} replaces $ (n) \sqcup \lambda$ with its conjugate partition
$\lambda' + (1^n)$; the conjugation arises because
the functor $L_\lambda$ in \cite{BCV} is our~$\nabla^{\lambda'}$.)
The proof in \cite{BCV}
 gives a bijection between the invariants 
in $\bigotimes^n (\bigwedge^m E)$ and $\bigotimes^n (\bigwedge^{m+1} E)$
for a Borel subgroup of $\GL(E)$.
Our
proof establishes a corresponding bijection between highest-weight vectors, with~$\bigwedge^m$
replaced with an arbitrary Schur functor.
We remark after this proof on the connection with the later
proof of~\eqref{eq:BCV} given in \cite[Lemma 3.2]{KahleMichalek}.
\end{samepage}

Specializing~\eqref{eq:BCV} by taking $\nu = (n)$ or $\nu=(1^n)$ gives
two results first proved in~\cite{Newell}. Since Newell's paper is not easy to read,
we explain the connection.
Taking the inner product of both sides of Theorem~\ref{thm:BCV-N} of \cite{Newell} with $s_\lambda$,
where $\lambda \in \Par(mn-k)$, gives
\begin{equation}\label{eq:Newell} 
 \bigl\langle (s_{(1^k)} \circ
s_{(m-1)})(s_{(n-k)} \circ s_{(m)}), s_\lambda \bigr\rangle
=\bigl\langle s_{(n)} \circ s_{(m)}, s_{(1^k)} s_\lambda \bigr\rangle
. \end{equation}
(This corrects a typographical error in \cite{Newell}: as can be seen
from the correctly stated and analogous Theorem 1A, $g_{(1^k)\xi\nu} \{\nu\}$
should be $g_{(1^k)\xi\nu} \{\xi\}$; note that, by definition,
$g_{(1^k)\xi\nu} = \langle s_{(1^k)} s_\xi , s_\nu \rangle$.) 
Taking $k=n$ we obtain 
$\langle s_{(1^n)} \circ s_{(m-1)}, s_\lambda \rangle
= \langle s_{(n)} \circ s_{(m)}, s_{(1^n)} s_\lambda \rangle$.
By Pieri's rule, $s_{(1^n)} s_\lambda$ is the sum of all the Schur functions labelled by partitions
obtained from $\lambda$ by adding~$n$ boxes, no two in the same row.
On the other hand, by Young's rule, $s_{(n)}$ is a summand of $s_{(1)}^n$, so the plethysm
$s_{(n)} \circ s_{(m)}$ is contained in $s_{(1)}^n \circ s_{(m)} = (s_{(1)}\circ s_{(m)})^{n} 
= s_{(m)}^n$. Another application of Young's rule now shows that
if $s_\lambda$ is a constituent of $s_{(n)} \circ s_{(m)}$ then $\ell(\lambda) \le n$. Hence
\begin{equation} \label{eq:Newell1}
 \langle s_{(1^n)} \circ
s_{(m-1)}, s_\lambda \rangle= \langle s_{(n)} \circ s_{(m)}, s_{\lambda + (1^n)}\rangle.
\end{equation} Similarly
Theorem 1A in \cite{Newell} implies that
\begin{equation} \label{eq:Newell2}
 \langle s_{(n)} \circ
s_{(m-1)}, s_\lambda \rangle
= \langle s_{(1^n)} \circ s_{(m)}, s_{\lambda + (1^n)} \rangle.
\end{equation}
These are the special case of the equivalent form of Theorem~\ref{thm:BCV-N}
stated in~\eqref{eq:BCV-Ntwisted} when $\mu = (1^m)$, $r = 1$
and $\nu = (n)$ or $\nu = (1^n)$. 

%


It follows from our Theorem 1, or by combining~\eqref{eq:Newell1} and~\eqref{eq:Newell2},
that $\langle s_{(n)} \circ s_{(m+2)}, s_{\lambda + (2^n)} \rangle = \langle s_{(n)} \circ s_{(m)},
s_\lambda \rangle$ for all $\lambda \in \Par(mn)$. This was proved by 
Dent in \cite[
Theorem 3.8]{DentThesis} using the symmetric group.

\subsubsection*{Theorem~\ref{thm:Brion}}
The special case of part (i)  of the theorem on page 354 of \cite{Brion}
when~$G$ is $\GL(E)$ asserts that if $\rho$ is any partition then
$\langle s_\nu \circ s_{\mu + N\rho}, s_{\lambda + Nn\rho} \rangle$
is a non-decreasing function of $N \in \N_0$. Part (ii) gives a condition in terms
of simple roots for the values
to stabilise: in the special case of $\GL(E)$, it becomes 
\[ \mu_i - \mu_{i+1} + 
N(\rho_i - \rho_{i+1}) \ge n(\mu_1 + \cdots + \mu_i) - (\lambda_1 + \cdots + \lambda_i) \]
for every $i$ such that $\rho_i > \rho_{i+1}$, as stated in \cite[page 362, Corollary 1]{Brion}. 
Part~(iii) gives a technical formula for the stable multiplicity.

Our Theorem~\ref{thm:Brion} is Brion's Theorem in the case $\rho = (1^r)$.
Brion's theorem for a general partition $\rho$ follows by
repeatedly applying our theorem to each column of $\rho$ in turn.
Brion's proof uses
$\GL(E)$-invariant vector bundles on the Grassmannian variety of full flags in $E$ and the long exact cohomology sequence.
The more elementary proof given here, which leads to a combinatorial
upper bound on the stable multiplicity (see Proposition~\ref{prop:stable}), is therefore of  interest.

Taking $\nu = (n)$, $\mu = (m)$ and $r=1$ in Theorem~\ref{thm:Brion} we obtain
$\langle s_{(n)} \circ s_{(m)}, s_\lambda\rangle \le \langle s_{(n)} \circ s_{(m+1)}, s_{\lambda+(n)}
\rangle$. This is Foulkes' Second Conjecture, stated as a working hypothesis 
at the end of \S 1 of \cite{Foulkes}, and proved by Brion in~\cite{Brion}.

\subsubsection*{Theorem~\ref{thm:Ikenmeyer}}
Proposition 4.3.4 of \cite{IkenmeyerThesis} is equivalent to the special case
of Theorem~\ref{thm:Ikenmeyer} when $\mu = (m)$. The proof in \cite{IkenmeyerThesis}
uses polynomial representations of $\GL(E)$, where $E$ is a complex vector space.
The key idea is to multiply highest-weight vectors in $\Sym^n (\Sym^m E)$ and
$\Sym^{\nstar}\! (\Sym^m E)$ to get a highest-weight vector in $\Sym^{n+\nstar} (\Sym^m E)$.
This generalizes  to prove Theorem~\ref{thm:Ikenmeyer}.
To motivate this proof we digress briefly 
to illustrate the geometric interpretation of this multiplication,
basing our discussion on the examples in \cite[\S 11.3]{FultonHarrisReps}. 
(This example is not logically essential.)

\begin{example}\label{ex:geometricFoulkes}
Let $E$ have basis $\u_1, \ldots, \u_d$.
Then $\Sym^2\! E$ has basis 
\[ \{ \u_i^2 : 1 \le i \le d \} \cup \{ 2\u_i\u_j : 1 \le i < j \le d\}. \]
Let $X_{ii} = (\u_i^2)^\star$ and $X_{ij} = (2\u_i\u_j)^\star$ be the corresponding elements
of the dual space $(\Sym^2 \! E)^\star$, regarded as a polynomial representation of~$\GL(E)$ by
the contravariant duality 
in \cite[\S 2.7]{ErdmannGreenSchockerGLn}. (Thus if $\rho(g)$ is the matrix
representing~$g$ in its action on a polynomial representation~$V$ of $\GL(E)$ then
$\rho(g^{\mathrm{tr}})^{\mathrm{tr}}$ represents~$g$ in its action on~$V^\star$.)
Let~$\mathcal{C}$ be the image of $E$ under
the map $E \rightarrow \Sym^2 \!E$ defined by 
$v \mapsto v^2$, so
\[ \mathcal{C} = \Bigl\{ \sum_{i=1}^d \alpha_i^2 e_i^2 + \!\!\sum_{1 \le i < j \le d} \!\!2\alpha_i\alpha_j e_i e_j
\,:\, \alpha_1, \ldots, \alpha_d \in \C \Bigr\}. \]
Thinking of $\mathcal{C}$ as an affine variety contained in $\Sym^2 \!E$, we see
that the vanishing ideal of~$\mathcal{C}$ in the coordinate ring $\mathcal{O}(\Sym^2\! E)$
contains $X_{11}X_{22} - X_{12}^2$ in degree~$2$. This is a highest-weight vector
for $\GL(E)$ of weight $(2,2)$, so by Proposition~\ref{prop:hw}(ii) it generates a 
submodule of $\mathcal{O}(\Sym^2\! E)$  isomorphic to $\nabla^{(2,2)}(E)$.
This submodule is the kernel of the map $\mathcal{O}(\Sym^2\! E ) 
\rightarrow \mathcal{O}(E)$ induced by 
restricting a coordinate function
on $\Sym^2\! E$ to $\mathcal{C}$ 
and then pulling it back to~$E$ using the squaring map \hbox{$E \rightarrow \mathcal{C}$}.
Thus
taking generators for the coordinate ring $\mathcal{O}(E)$ so that
$\mathcal{O}(E) = \C[Y_1,\ldots, Y_d]$, we have $X_{ij} \mapsto Y_i Y _j $.
This
defines a  homomorphism of $\GL(E)$-modules $\Sym^2 \bigl( (\Sym^2 \!E)^\star
\bigr) \rightarrow \Sym^4 (E^\star)$
with kernel $\nabla^{(2,2)}(E)$. 
Since $X_{11}^2 \in \mathcal{O}(\Sym^2\! E)_2$ maps to $Y_1^4
\in \mathcal{O}(E)_4$,
which is highest-weight of weight~$(4)$, this $\GL(E)$-homomorphism is surjective.
Since all irreducible $\GL(E)$-modules are self-dual under contravariant duality, we obtain
\[ \Sym^2(\Sym^2 \! E) \cong \Sym^4 \! E \oplus \nabla^{(2,2)}(E).\] 
Multiplying highest-weight vectors in the coordinate ring $\mathcal{O}(\Sym^2\! E)$, we see that for $r$, $s \in \N_0$, the product
$(X_{11}X_{22} - X_{12}^2)^r X_{11}^{s}$ vanishes on $\mathcal{C}$ with multiplicity $r$
and is highest-weight of weight $(2s+2r,2r)$. It follows that
\[ \Sym^n (\Sym^2 \!E) \cong \nabla^{2n}(E) \oplus \nabla^{(2n-2,2)}(E) \oplus \nabla^{(2n-4,4)}(E)
\oplus \cdots \oplus W \]
where if $\nabla^\lambda(E)$ appears in $W$ then $\ell(\lambda) \ge 3$. In particular,
if $d = 2$ then $W=0$, and every summand in the decomposition of $\Sym^n (\Sym^2 \!E)$ has
a simple geometric interpretation. The  decomposition for general $d$ is obtained in Example~\ref{ex:hwFoulkes}
below.
\end{example}

In \cite{deBoeck}, the first author constructed
explicit homomorphisms between modules for the symmetric group that prove
the special cases of Theorem~\ref{thm:Ikenmeyer} when $\mu = (1^m)$ and $n^\star = 1$
(and necessarily $\lambda^\star = (1^m)$) and,
subject to the conditions that $m$ is even and $\lambda$ has at most $2m$ parts, 
when $\mu = (1^m)$ and $\lambda^\star = (1^{m n^\star})$.

\subsubsection*{Theorem~\ref{thm:maxls}}
This theorem strengthens the main result of \cite{PagetWildonGeneralizedFoulkes}.
The proof in \cite{PagetWildonGeneralizedFoulkes} is entirely within the symmetric group, and
constructs an explicit homomorphism corresponding to each maximal partition $\lambda$
such that $s_\lambda$ appears in $s_\nu \circ s_\mu$. This requires a lengthy and quite
intricate argument, so again we believe that the shorter
proof presented here, which also gives a precise result on the multiplicity, is of interest.
As a corollary (see Corollary~\ref{cor:partitionWeight}) we show that if $T$ is a plethystic semistandard 
tableau of maximal weight for its shape, under the dominance order, then $\wt(T)$ is a partition. This fact
was mentioned in \cite{PagetWildonGeneralizedFoulkes}, where we noted that 
it has a non-trivial combinatorial proof using a 
variation on the Bender--Knuth involution  
(see \cite[page~47]{BenderKnuth}). 

Although it is not logically essential to the proof of Theorem~\ref{thm:maxls}, 
it is often useful in calculations (as seen already in Example~\ref{ex:ex}) that
the $\mu$-tableau entries of a plethystic semistandard tableau of shape $\mu^{(1^n)}$ and
maximal weight satisfy the following closure condition. Recall that $\unrhd$ denotes
the dominance order on partitions, defined before Theorem~\ref{thm:maxls}.


\begin{definition}\label{defn:closureAlt}
Let $\mathcal{T}$ be a set of semistandard $\mu$-tableaux. We say that~$\mathcal{T}$ is \emph{closed}
if whenever $t \in \mathcal{T}$ and $s$ is a semistandard $\mu$-tableau obtained from~$t$ by
changing a single entry $c$ to $c-1$, then $s \in \mathcal{T}$.
\end{definition}

\begin{proposition}\label{prop:closed}
Let $T$ be a plethystic semistandard tableau of shape $\mu^{(1^n)}$. If the weight of $T$ is maximal
in the dominance order for its shape then the set of $\mu$-tableau entries of $T$ is closed.
\end{proposition}

\begin{proof}
Let $\mathcal{T}$ be the set of $\mu$-tableau entries of $T$. Suppose that $\mathcal{T}$ is not closed,
and let $t \in \mathcal{T}$, $s \not\in \mathcal{T}$ and $c$ be as in Definition~\ref{defn:closureAlt}.
Let $S$ be the  plethystic semistandard tableau of shape $\mu^{(1^n)}$ obtained from $T$
by deleting $t$ and inserting $s$, and then reordering (if necessary) the $\mu$-tableau entries
within the single column of~$S$
so that~$S$ is column-standard in the total order $<$ on semistandard $\mu$-tableau in Definition~\ref{defn:totalorder}.
Then $S$ is a plethystic semistandard tableau of shape $\mu^{(1^n)}$ and
\[  \wt(S)_b = \begin{cases} 
\wt(T)_b+1 & \text{if $b = c-1$} \\
\wt(T)_b-1 & \text{if $b = c$} \\
\wt(T)_b & \text{otherwise.} \end{cases} 
\]
Therefore $\wt(S) \rhd \wt(T)$.
\end{proof}


\subsection*{Polynomial functors and highest-weight vectors}

Let $d \in \N$ and 
let~$E$ be a $d$-dimensional complex vector space. Let $V$ be an $D$-dimensional representation of $\GL(E)$
corresponding, under some choice of bases of $E$ and~$V$, to the homomorphism $\rho : \GL_d(\C) \rightarrow \GL_D(\C)$. Recall that~$V$ 
is a \emph{polynomial representation} of {degree $r$} if for all $a$, $b \in \{1,\ldots, D\}$
the matrix coefficient $\rho(g)_{ab}$ is a polynomial of degree $r$ in the matrix coefficients of 
$g \in \GL_d(\C)$.
Let $\GLEmod$ be the additive category of finitely generated
polynomial representations of $\GL(E)$ and let $\GLEmodr{r}$ be its full subcategory 
of representations of polynomial degree~$r$. 
For each $\lambda \in \Par(r)$ let $\nabla^\lambda : \GLEmod_p \rightarrow \GLEmod_{pr}$ 
be the Schur functor, as defined in~\S\ref{subsec:polyreps}. (Our construction
and proofs have some novel features, but this section will be background for most readers.)

By Proposition~\ref{prop:hw} every polynomial representation~$V \in \GLEmod$
decomposes as a direct sum of submodules each isomorphic to some $\nabla^\lambda(E)$.
Let $[W : \nabla^\lambda(E)]$ denote the number of irreducible summands of the 
polynomial representation
$W$ that are isomorphic to $\nabla^\lambda(E)$. 
By Proposition~\ref{prop:compositionPlethysm},
composition of Schur functors corresponds to plethysm of Schur functions. Hence, by Proposition~\ref{prop:hw}(iv),
\begin{equation}\label{eq:schurnabla} \langle s_\nu \circ s_\mu, s_\lambda \rangle
=  [ \nabla^\nu (\nabla^\mu E) : \nabla^\lambda E] 
 \end{equation}
for all partitions $\lambda$, $\mu$ and $\nu$ with at most $d$ parts.
Thus each of the main theorems has an equivalent restatement as a result on polynomial
representations of $\GL(E)$. 
To prove these restatements, we use the model for $\nabla^\nu\bigl( \nabla^\mu (E) )$ constructed
in \S\ref{sec:model} and
the following key fact about highest-weight vectors, as characterized in Lemma~\ref{lemma:hw} using
the Lie algebra action of $\gl(E)$, and proved in Proposition~\ref{prop:hw}: 
\emph{if $V$ is a polynomial $\GL(E)$-module 
then $V$ contains a highest-weight vector $v$; moreover, if the weight of $v$ is $\lambda$ then the submodule of $V$
generated by $v$ is isomorphic to $\nabla^\lambda(E)$}.

%
%
%

To illustrate the power of this property
we end this introduction  
by giving
a very short proof that if $\dim E \ge n$ then
\begin{equation} \label{eq:Foulkes}
\Sym^{n}(\Sym^2\! E ) \cong \!\bigoplus_{\lambda \in \Par(n)} \!\nabla^{2\lambda}(E). \end{equation}
By much more lengthy arguments, Boffi shows in \cite{Boffi} 
that, over an arbitrary field, $\Sym^n \hskip-0.5pt\Sym^2$ has a filtration by the functors $\nabla^\lambda$; 
he reports that this result had previously been obtained in~\cite{Abeasis}.
The analogous result for the symmetric
group was proved independently by the second author in~\cite{PagetFiltration}.
Boffi's result was generalized to arbitrary commutative rings in \cite{ChoiKimKoWon}.

\begin{example}\label{ex:hwFoulkes}
\newcommand{\W}[1]{W(#1)}
Let $E$ have basis $\u_1,\ldots, \u_n$.
For each $\s \in \{1,\ldots, n\}$
and each function $\sigma : \{1,\ldots, \s\}
\rightarrow \{1,\ldots, \s\}$, define
\[ w(\sigma) = (\u_{1}\u_{1\sigma})\ldots (\u_\s \u_{\s\sigma})
\in \Sym^\s (\Sym^2 E) \]
where $i \sigma$ is the image of $i$ under the function $\sigma$.
Let $S_\ell$ denote the symmetric group of all permutations of $\{1,\ldots, \ell\}$.
Define 
\[ \W{\s} = \sum_{\sigma \in \SymG_\s} w(\sigma) \sgn(\sigma) \]
for each $\s \in \N_0$. Since $\prod_{i=1}^s e_i e_{i\sigma} = \prod_{j=1}^\s e_{j\sigma^{-1}} e_j$
for $\sigma \in S_\ell$ we have $w(\sigma) = w(\pi)$ if $\pi = \sigma^{-1}$, and it is easily
seen that the converse also holds. In particular, the unique permutation $\sigma$ 
such that $w(\sigma) =
\u_1^2 \ldots \u_\s^2$ is the identity. Hence
the coefficient of $\u_1^2 \ldots \u_\s^2$ in $\W{\s}$ is~$1$, and $\W{\s} \not=0$.
For $c \in \{2,\ldots, \ell\}$, let
 $X^{(c)} \in \gl(E)$ be the Lie algebra element defined, as immediately before Lemma~\ref{lemma:hw},
by $X^{(c)} \cdot \u_c = \u_{c-1}$ and $X^{(c)} \cdot \u_b = 0$ if~$b\not= c$. 

Fix $c \in \{2,\ldots, \ell\}$ and
define $\df : \{1,\ldots,\s\} \rightarrow \{1,\ldots, \s\}$ by
$\df(c) = c-1$ and $\df(b) = b$ if $b \not= c$. Thus $\df(b) = c-1$ if and only if $b \in \{c-1,c\}$.
The action of $\gl(E)$ on symmetric powers is recalled immediately
before the proof of 
Proposition~\ref{prop:nabla}. Using this action and
$w(\sigma) = w(\sigma^{-1})$,
 we find that
\[ X^{(c)} \cdot w(\sigma) = (\u_{c-1} \u_{c\sigma}) \prod_{i \not= c} (\u_{i} \u_{i \sigma})
+ (\u_{c \sigma^{-1}} \u_{c-1}) \prod_{j \not= c} (\u_{j \sigma^{-1}}\u_j). \]
The summands are the products of all $\u_{b\sigma^{-1}\delta}\u_b$ and 
$\u_{b}{\u_{b\sigma\delta}}$, respectively.
Hence 
$X^{(c)} \cdot w(\sigma) = w(\sigma^{-1} \df) + w(\sigma \df)$ and so
\begin{equation} X^{(c)} \cdot \W{\s} = \sum_{\sigma \in S_\ell}  w(\sigma^{-1} \df) \sgn(\sigma) + \sum_{\sigma \in S_\ell}  w(\sigma \df)\sgn(\sigma) . \label{eq:XonW} \end{equation}
Let $\tau$ be the transposition $(c-1,c)$. Using that
$b \tau \df = b\df$ for each $b \in \{1,\ldots, \ell\}$ and so $\tau\df = \df$ we get
\[ \sum_{\sigma \in S_\s} w(\sigma \df) \sgn(\sigma) =
\sum_{\sigma' \in S_\s} w(\sigma' \tau \df) \sgn(\sigma' \tau) 
= -\sum_{\sigma' \in S_\s} w(\sigma' \df) \sgn(\sigma'). \]
Hence $\sum_{\sigma \in S_\s} w(\sigma \delta) \sgn(\sigma) = 0$. Similarly
$\sum_{\sigma \in S_\s} w(\sigma^{-1} \delta) \sgn(\sigma) = 0$. It now follows from~\eqref{eq:XonW} that
$X^{(c)} \cdot \W{\s} = 0$.
(This cancellation 
has an attractive combinatorial interpretation, shown in Figure~1 above.)
By Lemma~\ref{lemma:hw},~$\W{\s} \in \Sym^\s ( \Sym^2 E)$ is a highest-weight vector of weight~$(2^\s)$.

\begin{figure}[t]
\begin{center}
\scalebox{0.9}{\includegraphics{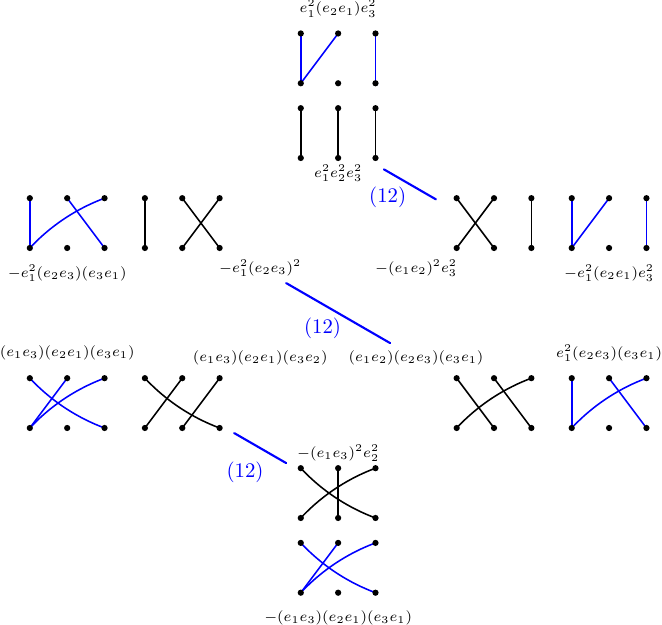}}
\end{center}
\vspace*{-6pt}
\caption{\small Example~\ref{ex:hwFoulkes} for $\ell=3$. Writing permutations $\sigma \in S_3$
in one-line form, the six summands in $w_3$ are $w(123)$, $w(213)$, $w(231)$, $w(321)$, $w(312)$, $w(132)$,
clockwise from the top. 
 Identifying $1$ and $2$ in the image of $\sigma$ (shown as the bottom row in the permutation
 diagrams)
the contributions from $\sigma$ and $\sigma (12)$ cancel. Cancelling pairs are connected
by edges marked $(12)$.
Thus $\sum_{\sigma \in S_3} w(\sigma \delta) \sgn(\sigma) = 0$.}
\end{figure}

Now take  $\lambda \in \Par(n)$.
Multiplying highest-weight vectors as in the proof of Theorem~\ref{thm:Ikenmeyer}
shows that if $a$ is the first part of $\lambda$
then 
\[ \W{\lambda'_1}\ldots \W{\lambda'_a} \in \Sym^n (\Sym^2\! E) \] 
is a highest-weight vector
of weight $2\lambda$. By the key fact stated before this example,
$\Sym^n (\Sym^2\! E)$ has $\nabla^{2\lambda}(E)$
as a summand.

We use the  Frobenius--Schur count of involutions to show
that there are no further summands. 
For $\lambda \in \Par(r)$,  
let $\chi^\lambda$ be the irreducible character of the symmetric
group $\SymG_r$ corresponding to $\nabla^\lambda(E)$
(and, by the characteristic isometry, to the Schur function~$s_\lambda$). By the count of involutions,
$\sum_{\lambda \in \Par(r)} \chi^\lambda(1) = T_r$, where $T_r$ is the number of 
permutations of order at most two in $\SymG_r$.
On the other hand, setting
\[ M_k = \bigl(\Sym^n (\Sym^2 \! E)\bigr) \otimes \smash{\bigwedge^k}\phantom{\rule{0pt}{16pt}} E, \]
it follows from Pieri's rule and the result already proved that
$M_k$ contains $\bigoplus_{\lambda} 
\nabla^\lambda (E)$, where the sum is over all partitions $\lambda \in \Par(r)$ having exactly~$k$ 
odd parts. By~\eqref{eq:charisoplethysm} in \S\ref{subsec:plethysm} below,
the character of the symmetric group $S_r$
corresponding to $\bigoplus_{2n + k = r} M_k$ is 
$\sum_{2n+k = r} (1_{C_2 \wr S_n} \times \sgn_{S_k})\!\ind_{C_2 \wr S_n \times S_k}^{S_r}$.
For $n \ge 1$, the degree of $(1_{C_2 \wr S_n} \times \sgn_{S_k})\!\ind_{C_2 \wr S_n \times S_k}^{S_r}$ is
the number of involutions in $S_{2n}$ having precisely~$k$ fixed points. Therefore this
character has degree~$T_r$ and
so $\bigoplus_{2n+k = r}M_k \cong \bigoplus_{\lambda \in \Par(r)}
\nabla^\lambda(E)$ and $M_0 = \Sym^n \bigl( \Sym^2 \! E\bigr)$ 
has no summands other than those already found.
\end{example}



\section{Background}\label{sec:background}

\subsection{Partitions and tableaux}\label{subsec:pts}
Let $\lambda$ be a partition of $r \in \N$. 
Recall that $\ell(\lambda)$ denotes the number of parts of $\lambda$.
The \emph{Young diagram} of $\lambda$ is the set $[\lambda] = \{ (i,j) : 1 \le i \le \ell(\lambda),
1 \le j \le \lambda_i\}$. Let $\B$ be a set.
A $\lambda$-\emph{tableau with entries from $\B$} is a function $t : [\lambda] \rightarrow \B$.
We write the image of $(i,j)$ under $t$ as $t_{(i,j)}$. 
If $t_{(i,j)} = b \in \B$ then we say that $b$ 
is the \emph{entry} of $t$ in row $i$ and column $j$.
When $\B$ is a set of natural numbers,
this corresponds to the usual diagrammatic representation of tableaux (see \S\ref{subsec:polyreps} for
a small example). 

Now suppose that $\B$ is totally ordered by an order denoted $<$.
We say that a tableau~$t$ with entries from $\B$ is \emph{row-semistandard} if its rows are weakly increasing
from left to right, \emph{column-standard} if its columns are strictly increasing from top to bottom,
both under the order $<$. We say that $t$ is
\emph{semistandard} if it is both row-semistandard and column-standard. 
The terms
\emph{row-standard} and \emph{standard} are defined analogously, requiring in addition
 that the rows are strictly
increasing. 
Let $\SSYT_\B(\lambda)$ denote the set of semistandard $\lambda$-tableaux with entries from~$\B$.
When $\B \subseteq \N$, we refer to elements of $\SSYT_\B(\lambda)$ simply as semistandard $\lambda$-tableaux.
If $\B = \{1,\ldots, d\}$ and $t$ has exactly $\beta_b$ entries 
for each $b \in \{1,\ldots,d\}$, then we say that $t$ has \emph{content}~$\beta$.
Let $\SSYT(\lambda,\beta)$ denote the set of semistandard $\lambda$-tableaux 
of content $\beta$.
Let $t^{\lambda}$ be the unique element of $\SSYT(\lambda,\lambda)$, as defined by
$t^\lambda_{(i,j)} = i$ for
each \hbox{$(i,j) \in [\lambda]$}. In the total order on column-standard $\lambda$-tableau in the following
definition,
$t^\lambda$ is the least element.

\begin{definition}\label{defn:totalorder}
Given column-standard $\lambda$-tableaux $t$ and $u$ with entries
in a totally ordered set, set $t < u$ if and only if in the rightmost column
that differs between $t$ and $u$, the greatest entry not appearing in both columns lies in $u$.
\end{definition} 

In particular, if $\mu$ and $\nu$ are partitions then the semistandard $\mu$-tableaux are totally
ordered by $<$. A plethystic semistandard tableau of shape $\mu^\nu$, as defined in Definition~\ref{defn:pssyt},
is a semistandard $\nu$-tableau with entries from $\SSYT_{\N}(\mu)$.


The symmetric group on $[\lambda]$ acts on $\lambda$-tableaux on the right
by place permutation: thus if $\sigma \in \SymG_{[\lambda]}$ and $t$ is a $\lambda$-tableau 
then $(t  \sigma)_{(i,j)} = t_{(i,j)\sigma^{-1}}$. Thus the entry in position $(i,j)$ of~$t$
is found in position $(i\sigma,j\sigma)$ of $t\sigma$. Let 
$\CPP(\lambda)$ denote the Young subgroup
of $\SymG_{[\lambda]}$ having as its orbits the 
columns of~$[\lambda]$.

\subsection{Schur functions and plethysm}\label{subsec:plethysm}
Let $\Lambda$ be the ring of symmetric functions as defined in \cite[\S 7]{StanleyII}.
Let $\lambda \in \Par(r)$.
Given a $\lambda$-tableau~$t$ of content $\beta$, let~$x^t$ denote the
monomial $x_1^{\beta_1} \ldots x_{\ell(\beta)}^{\beta_{\ell(\beta)}}$.
The \emph{Schur function} $s_\lambda \in \Lambda$ is defined by
$s_\lambda = \sum_{t \in \SSYT_\N(\lambda)} x^t$. 
By~\cite[Theorem 7.10.2]{StanleyII}, $s_\lambda$ is a symmetric function. Thus
\begin{equation}
\label{eq:Schur} s_\lambda = \sum_{\kappa \in \Par(r)} |\SSYT(\lambda,\kappa)| \mon_\kappa
\end{equation}
where $\mon_\kappa = x_1^{\kappa_1}\ldots \smash{x_{\ell(\kappa)}^{\kappa_{\ell(\kappa)}}} + \cdots$ 
denotes the monomial symmetric function corresponding to $\kappa$.
For example, $s_{(2)} = x_1^2 + x_2^2 + \cdots + x_1x_2 + x_1x_3+x_2x_3 + \cdots 
= \mon_{(2)} + \mon_{(1,1)}$.

Given $f(x_1,x_2,\ldots ) \in \Lambda$, the plethysm $f \circ s_\lambda$  may be defined by substituting
the monomials $x^t$, where $t$ ranges over all semistandard $\lambda$-tableaux with entries from $\N$,
for the variables $x_1, x_2,\ldots $ of $f$. The general definition of $f \circ g$ is given in
\cite[Ch.~7, Appendix 2]{StanleyII}, 
\cite[Ch.~1, Appendix A]{MacDonald} or \cite{LoehrRemmel}.
For example,
the combinatorial analogue of the case $n=2$ of~\eqref{eq:Foulkes} is
\[ s_{(2)} \circ s_{(2)} = s_{(2)}(x_1^2, x_1x_2, \ldots )
= \mon_{(4)} + \mon_{(3,1)} + 2\mon_{(2,2)} = s_{(4)} + s_{(2,2)}.\]
The coefficient of $\mon_{(2,2)}$ is $2$ since $x_1^2x_2^2$ may be
obtained as both $(x_1^2)(x_2^2)$ 
and $(x_1x_2)^2$ when multiplying out $\smash{s_{(2)}(x_1^2,x_1x_2,\ldots)}$.

There is an involutory ring homomorphism 
$\omega : \Lambda \rightarrow \Lambda$ defined by $\omega(s_\lambda) = s_{\lambda'}$. 
We call $\omega$ the \emph{sign twist}. Its effect on plethysms is as follows.

\begin{lemma}\label{lemma:signtwist}
Let $\mu$ and $\nu$ be partitions. If $\mu$ is a partition of $m$ then
\[ \omega(s_\nu \circ s_\mu) = \begin{cases} s_\nu \circ s_{\mu'} & \text{if $m$ is even} \\
s_{\nu'} \circ s_{\mu'} & \text{if $m$ is odd.} \end{cases} \qedhere \]
\end{lemma}

\begin{proof}
See \cite[Ch.~I, Equation~(2.7)]{MacDonald}.
\end{proof}

In Examples~\ref{ex:ex} and \ref{ex:hwFoulkes} we used the characteristic
isometry, which sends the irreducible character $\chi^\lambda$ of~$\SymG_r$
to the Schur function $s_\lambda$ where $\lambda \in \Par(r)$. 
By \cite[A2.8]{StanleyII} or 
\cite[Ch.~1, Appendix A, (6.2)]{MacDonald}, if $\mu \in \Par(m)$ and $\nu \in \Par(n)$
then under this isometry,
\begin{equation}
  \label{eq:charisoplethysm} 
  \bigl( (\chi^{\mu})^{\widetilde{\times n}} \Inf_{S_n}^{S_m \wr S_n} \!\chi^\nu \bigr) \Ind_{S_m 
  \wr S_n}^{S_{mn}} \mapsto
    s_{\nu} \circ s_{\mu}
\end{equation}
where $(\chi^{\mu})^{\widetilde{\times n}}$ is the character of 
the irreducible representation of $\SymG_m \wr \SymG_n$
defined in \cite[(4.3.8)]{JK} and $\Inf_{S_n}^{S_m\wr S_n}$ is the inflation map.
Thus all the  main theorems have immediate translations
into results on characters of symmetric groups.
We 
 shall not use the characteristic isometry any further below.


\subsection{Polynomial representations of $\GL_d(K)$.}
\label{subsec:polyreps}

In the following three subsections
we construct the Schur functors $\nabla^\lambda$ 
used in the proofs of the main theorems. We end with Remark~\ref{remark:James}
which explains the connection with an earlier construction due to James \cite[Ch.~26]{James}.
Let $K$  be a commutative ring, let~$\B$ be a totally ordered set as in \S\ref{subsec:pts},
and let $V$ be a free $K$-module with basis $\{v_b : b \in \B \}$.
Except in the proof of Lemma~\ref{lemma:GLgarnir}, the case $K = \C$ suffices. The set
 $\B$ is either $\{1,\ldots, d\}$
or, when we compose Schur functors, 
the set of semistandard tableaux with entries from $\{1,\ldots, d\}$,
with the order from Definition~\ref{defn:totalorder}.

Fix $\lambda \in \Par(r)$. Recall that 
 $\CPP(\lambda)$ is defined after Definition~\ref{defn:totalorder}.

\begin{definition}\label{defn:polytabloids}
Let $\Sym^\lambda \V =  \bigotimes_{i=1}^{\ell(\lambda)} \Sym^{\lambda_i} \V$.
Given a $\lambda$-tableau $t$ with entries from $\B$, the \emph{$\mathrm{GL}$-tabloid}
corresponding to $t$ is the element $f(t) \in \Sym^\lambda \V$ defined by
\[ f(t) = \bigotimes_{i=1}^{\ell(\lambda)} \prod_{j=1}^{\lambda_i} v_{t_{(i,j)}} \in \Sym^\lambda \V.\]
The $\emph{$\mathrm{GL}$-polytabloid}$ corresponding to $t$ is
\[ F(t) = \sum_{\sigma \in \CPP(\lambda)} f(t  \sigma)\sgn(\sigma) \in \Sym^\lambda \V. \]
We define $\nabla^\lambda(V)$ 
to be the $K$-submodule of $\Sym^\lambda \V$ spanned by
the $\mathrm{GL}$-polytabloids $F(t)$ for $t$ a $\lambda$-tableau with
entries from $\B$. 
\end{definition}

Since $f(t) = f(t')$ if and only if the rows of $t$ and $t'$ are equal as multisets, 
\begin{equation}\label{eq:symBasis} 
\{ f(t) : \text{ $t$ a row-semistandard $\lambda$-tableau with entries from $\B$} \} \end{equation}
is a basis of  $\Sym^\lambda(V)$.
It is also useful to note that if $\sigma \in \CPP(\lambda)$ then 
\begin{equation} \label{eq:column} F(t  \sigma) = F(t) \sgn(\sigma). \end{equation}
In particular, $F(t) = 0$ if  $t$ has a column with a repeated entry, and
so  $\nabla^\lambda(V) = 0$ if $\ell(\lambda) > \dim V$, as used in Example~\ref{ex:geometricFoulkes}.
A potential trap is that
 $F(t)$ depends on the tableau $t$, not just on the $\GL$-tabloid $f(t)$. 
 For example, if $\B = \{1,2,3\}$ and
\[ t = \young(13,21)\, , \quad t' = \young(13,12) \]
then $f(t) = f(t')$ 
but $F(t') = 0$  
whereas
$F(t) = v_1v_3 \otimes v_2v_1 - v_2v_3 \otimes v_1^2 - v_1^2 \otimes v_2v_3 + v_2v_1 \otimes v_1v_3 \not=0$.
It is clear that $\nabla^\lambda(V)$ is functorial in $V$, so $\nabla^\lambda$
is a endofunctor of the category of free $K$-modules of finite rank.

Postponing the action of the general linear group 
for the moment, we find an explicit basis for $\nabla^\lambda(V)$,
introducing two results that are critical to the proofs of the main theorems.
The following lemma is the analogue
for $\GL$-polytabloids 
of part of the proof of Theorem 8.4 in \cite{James}. There are some subtle differences between the proofs
because of our use of place permutations.

\begin{lemma}\label{lemma:GLgarnir}
Let $1 \le j < j' < \lambda_1$ and let $1 \le i \le \lambda'_{j'}$.
Define subsets of $[\lambda]$ by $A_\lambda(i,j) = \{(i,j),\ldots,(\lambda_j',j)\}$ and 
$B_\lambda(i,j') = \{(1,j'),\ldots,(i,j')\}$. If $t$ is a $\lambda$-tableau then
\[ \sum_{\tau} F(t  \tau) \sgn(\tau) = 0 \]
where the sum is over all $\tau \in  \SymG_{A_\lambda(i,j) \cup B_\lambda(i,j')}$.
\end{lemma}

\begin{proof}
Let $A$ and $B$ denote $A_\lambda(i,j)$ and $B_\lambda(i,j')$, respectively. By definition,
the left-hand side is 
$\sum_{\tau \in \SymG_{A \cup B}} \!\sum_{\sigma \in \CPP(\lambda)} f\bigl( (t\tau) \sigma \bigr) \sgn(\tau) \sgn(\sigma)$.
Therefore it suffices to
show that
$\sum_{\tau \in \SymG_{A \cup B}} f\bigl( (t  \tau) \sigma\bigr) \sgn(\tau) = 0$
for each $\sigma \in \CPP(\lambda)$. 
Since $|A \cup B| = \lambda_j' + 1$, there exist boxes $(h,j)\sigma \in A\sigma$ 
and $(h',j')\sigma \in B\sigma$ such that $(h,j)\sigma$ and $(h',j')\sigma$ are in the same row of $[\lambda]$. 
Let $\delta = \bigl( (h,j)\sigma, (h',j')\sigma \bigr) \in \SymG_{A\sigma \cup B\sigma}$
be the transposition swapping these boxes. 
Let $\theta_1, \ldots, \theta_\ell$ be representatives for the left cosets of $\langle \delta\rangle$
in $\SymG_{A \sigma \cup B\sigma}$. Thus
$\SymG_{A\sigma \cup B\sigma} = \bigcup_{c=1}^\ell \theta_c \langle \delta \rangle$
and
\begin{align*}
\sum_{\tau \in \SymG_{A \cup B}} \!\!\! f\bigl( (t  \tau) \sigma\bigr) \sgn(\tau) 
 &= \!\!\! \sum_{\tau \in \SymG_{A \cup B}} \!\!\! f\bigl( (t  \sigma)  (\sigma^{-1}\tau \sigma) \bigr) \sgn(\tau) \\
 &= \sum_{\tau^\star \in \SymG_{A\sigma \cup B\sigma}} \!\!\! f (t \sigma \tau^\star) \sgn(\tau^\star) \\
 &= \!\sum_{c=1}^\ell \bigl( f(t \sigma   \theta_c) 
- f(t \sigma \theta_c  \delta) \bigr) \sgn(\theta_c) \\
&= 0 \end{align*}
where the final equality holds because $\delta$
swaps two boxes in the same row of $[\lambda]$, 
and so the tableaux $t \sigma \theta_h$ and $(t \sigma \theta_h) \delta$ 
have equal rows.
%
\end{proof}

%

If $\tau \in S_{A_\lambda(i,j)} \times S_{B_\lambda(i,j')}$ then,
by~\eqref{eq:column}, $F(t \tau) \sgn(\tau) = F(t)$. 
Let $\phi_1, \ldots, \phi_\ell$ be representatives for the left cosets
of $\SymG_{A_\lambda(i,j)} \times \SymG_{B_\lambda(i,j')}$
in $\SymG_C$, where $C = A_\lambda(i,j)\cup B_\lambda(i,j')$, chosen so that $\phi_1 = \id$. Thus
$S_C = \bigcup_{c=1}^\ell 
\phi_c(\SymG_{A_\lambda(i,j)} \times \SymG_{B_\lambda(i,j')})$.
By~\eqref{eq:column}, $F(t \phi_c \tau) = F(t \phi_c) \sgn(\tau)$ for each
$c$ and each $\tau \in S_{A_\lambda(i,j)} \times S_{B_\lambda(i,j')}$.
Therefore Lemma~\ref{lemma:GLgarnir} implies that
\[ |A_\lambda(i,j)|!\, |B_\lambda(i,j')|! \sum_{c=1}^\ell F(t \phi_c) \sgn(\phi_c)  = 0.\]
When $K = \Z$ we may cancel the factorials since $\nabla^\lambda(V)$ is a submodule
of the free $\Z$-module $\Sym^\lambda (V)$.
Thus the relation
\begin{equation}
\label{eq:snake}
F(t) = -\sum_{c=2}^\ell F(t \phi_c) \sgn(\phi_c).
\end{equation}
holds over an arbitrary commutative ring $K$.
We call~\eqref{eq:snake} a \emph{snake relation}, because of the shape formed by the boxes in~$A_\lambda(i,j) \cup
B_\lambda(i,j')$
when $j' = j+1$.
It is critical to the proofs of 
Theorem~\ref{thm:BCV-N},~\ref{thm:Brion} and~\ref{thm:maxls}.

It is convenient to choose the coset
representatives $\phi_1, \ldots, \phi_\ell$ so that each $\phi_c$
is a product of transpositions swapping boxes in $A_\lambda(i,j)$ and $B_\lambda(i,j)$,
preserving the relative vertical order of boxes in each set.

\begin{example} Let $\lambda = (2,2,1)$.
The snake relation for $A_{\lambda}(2,1) = \{(2,1),(3,1)\}$ and $B_{\lambda}(2,2) = \{(1,2),(2,2)\}$
and a $\lambda$-tableau $t$.
has five summands on its right-hand side. Depending on $t$,
some of these summands may vanish. For example
\[ F\left( \;\young(11,32,4)\, \right) = F\left( \;\young( 11,23,4)\, \right) + F \left( \;\young(11,34,2)\, \right) \]
because, taking coset representatives as suggested
above, the tableaux obtained by the transpositions
$(1,2) \leftrightarrow (2,1)$ and $(1,2) \leftrightarrow (3,1)$, and the double transposition
$(1,2) \leftrightarrow (2,1)$, $(2,2) \leftrightarrow (3,1)$ 
have a repeated~$1$ in their first column. 
By~\eqref{eq:column},
applying the transposition $(2,1) \leftrightarrow (3,1)$ 
to the second summand switches its sign and
expresses the left-hand side as a linear
combination of semistandard $\GL$-polytabloids.
\end{example}

\begin{corollary}\label{cor:snake}
If $t$ is a $\lambda$-tableau with entries from $\B$
then $F(t)$ may be expressed as a 
$K$-linear combination of $\GL$-polytabloids $F(s)$ for semistandard $\lambda$-tableaux $s$ by applying
finitely many snake relations.
\end{corollary}

\begin{proof}
By~\eqref{eq:column}, we may assume that $t$ is column-standard.
If $t$ is not standard then there exist $(i,j), (i,j+1) \in [\lambda]$
such that $t_{(i,j)} > t_{(i,j+1)}$. Let $A$, $B \subseteq [\lambda]$ be as in Lemma~\ref{lemma:GLgarnir},
taking $j' = j+1$. By~\eqref{eq:snake}, $F(t) = -\sum_{c=2}^\ell F(t\phi_c) \sgn(\phi_c)$ 
where each~$\phi_c$ swaps certain boxes in $A$ with certain boxes,
necessarily having smaller entries, in $B$.
Thus $t < t\phi_c$ for each $\phi_c$, where $<$ refers to the order in Definition~\ref{defn:totalorder}. 
The result now follows by induction.
\end{proof}

We say that $F(t)$ is \emph{straightened} by snake relations.
A related result to~\eqref{eq:snake}
 gives some control over the $F(s)$ that may appear in the straightening of $F(t)$;
it is needed in the proof of Theorem~\ref{thm:Brion}. To state it, we require two further
orders. 

Define a \emph{composition} of $n \in \N_0$
of length $\ell$ to be an element $\beta \in \N_0^\ell$
such that $\sum_{i=1}^\ell \beta_i = n$. 
We set $\ell(\beta) = \ell$. We extend the dominance order from partitions to compositions
in the obvious way, by setting 
$\beta \,\unrhd\, \gamma$ if $\beta_1 + \cdots + \beta_i
\ge \gamma_1 + \cdots + \gamma_i$ for all $i \in \N$.
 (As usual, if $i$ exceeds
the number of parts of $\beta$ or $\gamma$ then the corresponding part is taken to be $0$.)

\begin{definition}\label{defn:dom}
 Given a row-semistandard tableau $t$
with entries from $\{1,\ldots, d\}$, and $b \in \{1,\ldots, d\}$,
let $t^{\le b}$ be the composition $\gamma$ defined by $\gamma_i = \bigl| \{j : 1 \le j \le \lambda_i,
t_{(i,j)} \le b \} \bigr|$. 
If $u$ is a row-semistandard tableau with entries also from $\{1,\ldots, d\}$
and of
the same shape as $t$, we say that~$t$ \emph{dominates} $u$,
and write $t \unrhd u$ if $t^{\le b} \unrhd u^{\le b}$ for all $b \in \{1,\ldots, d\}$.
We extend this order to tableaux with entries from an arbitrary totally ordered set~$\B$ 
by the unique order-preserving bijection between $\{1,\ldots,|\B|\}$ and $\B$.
\end{definition}

\begin{definition}\label{defn:rowstraightening}
Given a tableau $t$ with entries from $\N$, let $\overline{t}$ be the row-semistandard
tableau obtained by sorting the rows of $t$ into non-decreasing order.
\end{definition}



\begin{proposition}\label{prop:boundGL}
Let $t$ be a column-standard $\lambda$-tableau with entries from $\{1,\ldots,d\}$.
Then $\overline{t}$
is semistandard and
\[ F(t) = F(\overline{t}) + w \]
where $w$ is an integral linear combination of $\GL$-polytabloids $F(s)$ for
semistandard $\lambda$-tableaux $s$ such that $\overline{t} \rhd s$.
\end{proposition}

\begin{proof}
We reduce to the analogous result for tableaux with distinct entries proved in \cite[Proposition 4.1]{WildonBound}. 
Let $t$ have content $\beta$. Let $t_\star$ be
the tableau obtained from~$t$ by replacing each of the $\beta_b$ entries
of $t$ equal to $b$ with a \emph{symbol} $b^{(1)}, \ldots, b^{(\beta_b)}$, for 
each $b \in \{1,\ldots,d\}$.
We say that $b^{(i)}$ has \emph{number} $b$ and \emph{exponent} $i$. We order symbols
lexicographically, by number then exponent,
 so $b^{(i)} < c^{(j)}$ if and only if $b < c$ or $b=c$ and $i < j$. Thus $t_\star$ has distinct
 entries and is column-standard.
By Proposition~4.1 of \cite{WildonBound}, $\overline{t_\star}$ is standard.
Let $V_\star$ be the free $K$-module with basis vectors $v_{b^{(i)}}$ in bijection with symbols. 
Again by \cite{WildonBound}, now working 
in $\nabla^\lambda(V_\star)$, we have $F(t_\star) = F(\overline{t_\star}) + w$ where
$w$ is an integral linear combination of $\GL$-polytabloids $F(s)$ for standard~$\lambda$-tableaux
$s$ (each
having symbol entries).
The proposition now follows from functoriality: the quotient map $V_\star \rightarrow V$
sending $v_{b^{(i)}}$ to $v_b$ for each symbol $b^{(i)}$ corresponds to
replacing each symbol with its number.
\end{proof}

Corollary~\ref{cor:snake} also does most of the work to prove a well-known
basis theorem
for $\nabla^\lambda(V)$. We include the details since the following 
lemma is also needed
in the proof of Theorem~\ref{thm:Brion}.
As a notational convenience, we extend the dominance order~$\rhd$ on row-semistandard tableaux to
$\GL$-tabloids by setting $f(t) \unrhd f(u)$ if and only
if~$\overline{t} \unrhd \overline{u}$.


%

\begin{lemma}\label{lemma:dom}
Let $t$ be a column-standard $\lambda$-tableau with entries
from $\B$. Let $F(t) = f(t) + w$ where $w \in \Sym^\lambda(V)$.
If $u$ is a row-semistandard $\lambda$-tableau such that
$f(u)$ appears with non-zero coefficient when $w$ is written in the canonical basis~\eqref{eq:symBasis}
of $\Sym^\lambda(V)$, then
$f(t) \rhd f(u)$.
\end{lemma}

\begin{proof}
By definition $F(t) = \sum_{\tau \in \CPP} f(t\tau) \sgn(\tau)$. From the identity permutation
we get the summand $f(t)$. Suppose that $\tau$ is not the identity permutation. Then 
there
exist boxes $(i,j)$ and $(i',j) \in [\lambda]$ such that $i < i'$ and 
$(i,j)\tau = (k,j)$ and $(i',j)\tau = (k',j)$ with $k > k'$. We say that such
boxes form a \emph{column inversion} of~$\tau$. 
Let $\delta = \bigl( (i,j), (i',j) \bigr)$.
Let $c = t_{(i,j)}$ and
let $c' = t_{(i',j)}$. Note that since~$t$ is column standard, $c < c'$.
It is easily seen that if $c \le b < c'$
then the Young diagrams of $\overline{t}^{\le b}$ and \smash{$\overline{t\delta}^{\le b}$}
differ by a single box, moved down from row $i$ (the row of $c$ in $t$) to row $i'$ (the row of $c'$ in~$t$). 
For all other
$b$ we have \smash{$\overline{t}^{\le b} = \overline{t\delta}^{\le b}$}. 
Hence $\overline{t \delta} \lhd \overline{t}$. By induction on the number
of column inversions, we have
$\overline{t\tau} \unlhd \overline{t\delta}$. Hence $\overline{t\tau} \lhd \overline{t}$.
Therefore $f(t) \rhd f(t\tau)$, as required.
\end{proof}

We note this proof is essentially the same as that of Lemma~8.3 in \cite{James}, modified to use
place permutations. It is also possible
to reduce to this result by distinguishing equal entries of~$t$ by formal symbols,
as in the proof of Proposition~\ref{prop:boundGL}.


%
\begin{proposition}\label{prop:semistandardBasis}
The set $\{F(s) : s \in \SSYT_{\B}(\lambda) \}$ 
is a  $K$-basis
for~$\nabla^\lambda(V)$. 
\end{proposition}


\begin{proof}
By Corollary~\ref{cor:snake}, $\nabla^\lambda(V)$ is spanned by $\{F(s) : s \in \SSYT_\B(\lambda) \}$.
Let $v = \sum_{s \in \SSYT_{\B}(\lambda)} \alpha_s F(s)$ where 
not every coefficient is zero. 
Take~$s$ maximal in the dominance order on semistandard $\lambda$-tableaux such that $\alpha_s \not= 0$.
By Lemma~\ref{lemma:dom}, the coefficient of $f(s)$ in $F(s)$ is $1$.
Again by this lemma, if $s'$ is a semistandard $\lambda$-tableau other than $s$
such that $\alpha_{s'} \not= 0$, then $f(s') \unrhd f(t')$ 
for every $f(t')$ appearing in $F(s')$.
Hence, by maximality of $s$, we have 
$f(t') \not\!\hskip-1.75pt\unrhd\hskip3pt f(s)$.
Therefore the coefficient of $f(s)$ in $v$ is $\alpha_s$, and so $v\not=0$.
\end{proof}


\subsection{Action of $\GL(E)$}\label{subsec:GLE}
We now suppose that $K$ is an infinite field and that 
$E$ is a $d$-dimensional $K$-vector space.
Suppose that $V$ is a polynomial
$\GL(E)$-module with basis, as in the previous section, $\{v_b : b \in \B\}$.
Let~$t \in \SSYT_\B(\lambda)$. The action of $g \in \GL(E)$
on $F(t) \in \nabla^\lambda(V)$, where $t$ is a $\lambda$-tableau with entries from~$\B$, is
determined by the multilinear construction in Definition~\ref{defn:polytabloids}.
The following method is convenient in calculations:
formally replace each entry $b$ in $t$ with $gv_b$,
expressed as a $K$-linear combination of~$\{ v_b : b \in \B \}$,
and then expand multilinearly.
For example, suppose that~$E$ is $3$-dimensional and $V$ is the natural
representation $E$, so we take $\mathcal{B} = \{1,2,3\}$.
Thinking of $\GL(E)$ as $3 \times 3$ invertible matrices, let 
\[ g = \left(\begin{matrix} \alpha & 0 & 0 \\ \delta & \beta & 0 \\ \epsilon & 0 & \gamma \end{matrix}\right) \in \GL(E).\]
Then in its action on $\nabla^{(2,2)}(V)$ we have
\begin{align*}
g F(\, \young(12,31)\, ) &\! = F\bigl(\, {\setlength{\arrayrulewidth}{0.65pt}
\begin{tabular}{|c|c|}\hline $\alpha v_1 + \delta v_2 + \epsilon v_3$ & $\beta v_2$ \\ \hline
$\gamma v_3$ & $\alpha v_1 + \delta v_2 + \epsilon v_3$ \\ \hline
\end{tabular}}\, \bigr) \\ 
&= \alpha^2 \beta \gamma F\bigl(\, \young(12,31)\, \bigr)
\cp \alpha \beta\gamma\delta F\bigl(\, \young(22,31)\, \bigr) 
\cp \alpha \beta \gamma\epsilon  F\bigl(\, \young(12,33)\, \bigr)
\cp \beta \gamma \delta\epsilon F\bigl(\, \young(22,33)\, \bigr)
\end{align*}
where the first line should be interpreted entirely formally. One may
then use snake relations to express the right-hand side in the standard basis of $\nabla^{(2,2)}(V)$.

 In this example, $V$ had polynomial degree $1$.
In general, if the $\GL(E)$-module~$V$
has degree $p$, then, identifying $\GL(E)$ with $\GL_d(\C)$, 
the coefficients in $g F(t)$ have degree $pr$ in the matrix coefficients of $g \in \GL_d(\C)$.
Thus if $|\lambda| = r$ then
 $\nabla^\lambda : \GLEmod_p \rightarrow \GLEmod_{pr}$ is an endofunctor of $\GLEmod$.
(This was seen for the decomposition of $\Sym^{n} \bigl( \Sym^2 \! E \bigr)$
in Example~\ref{ex:geometricFoulkes}.)
For a further example, if $\dim E = 2$, $V = \Sym^2\! E$, $\lambda = (2,1)$
and  $v_1 = e_1^2, v_2=e_1e_2, v_3 = e_2^2$ then working in $\nabla^{(2,1)}\bigl( \Sym^2 (E) \bigr)$
we have
\begin{align*} \left( \begin{matrix} \alpha & \beta \\ \gamma & \delta \end{matrix} \right)
F\left( \, \young(13,2) \, \right) 
&= F\left(\hskip0.75pt
\raisebox{-15pt}{\scalebox{0.9}{$\ltableau{4.75cm}{0.17cm}{$\alpha^2 v_1 + 2\alpha\gamma v_2 + \gamma^2 v_3$}{$\beta^2v_1 + 2\beta\delta v_2 + \delta^2 v_3$}{$\alpha\gamma v_1 + (\alpha\delta + \beta\gamma)v_2 + \beta\delta v_3$}$}}\hskip0.75pt \right) \\
& = \alpha^2\beta^2(\alpha\delta+\beta\gamma) F(\left(\hskip0.5pt \young(11,2)\hskip0.5pt\right) +
 2\alpha^2 \beta\delta(\alpha \delta+\beta\gamma) F\left(\hskip0.5pt \young(12,2) \hskip0.5pt \right)
 \\ & \hskip0.75in + \alpha^2\delta^2(\alpha\delta+\beta\gamma)
F\left(\hskip0.5pt \young(13,2) \hskip0.5pt \right) + \cdots .
\end{align*}
where the coefficients have degree $6$. 

\subsection{Action of $\gl(E)$ and highest-weight vectors}\label{subsec:weights}
As in the previous subsection,
let $K$ be an infinite field and let $E$ be a $d$-dimensional $K$-vector space. 
Fix a basis $e_1, \ldots, e_d$ of $E$, and use it to identify $\GL(E)$ with $\GL_d(K)$.
Recall that if $V$ is a polynomial representation of $\GL(E)$ and~$\beta$
is a composition with $\ell(\beta) = d$ then a non-zero vector
$v \in V$ is a \emph{weight vector} of \emph{weight}~$\beta$ if 
\begin{equation}\label{eq:weight}
g v = g_{11}^{\beta_1} \ldots g_{dd}^{\beta_d}\hskip0.5pt v 
\end{equation}
for all diagonal matrices $g \in \GL_d(K)$.
Let $V_\beta$ be the subspace of $V$ of weight vectors of weight $\beta$, together with $0$.
The \emph{formal character} of a polynomial representation~$V$ of $\GL(E)$
is the polynomial
\[ \Phi_V(x_1,\ldots,x_d) = \sum_{\beta} \dim (V_\beta) x_1^{\beta_1} \ldots x_d^{\beta_d} \]
where the sum is over all compositions $\beta$ such that $\ell(\beta) = d$.




\begin{lemma}\label{lemma:nabla} {\ }
\begin{thmlist}
\item If $s \in \SSYT(\lambda,\beta)$ then $F(s) \in \nabla^\lambda(E)$ has weight $\beta$.
\item The formal character of $\nabla^\lambda(E)$ is $s_\lambda(x_1,\ldots,x_d)$.
\end{thmlist}
\end{lemma}

\begin{proof}
If $g \in \GL(E)$ is a diagonal matrix then $g f(s) = g_{11}^{\beta_1} \ldots g_{dd}^{\beta_d} f(s)$.
Since all the $\lambda$-tableaux appearing in $F(s)$ have the same content as $s$,
it follows that $F(s)$ has weight~$\beta$, proving (i).
Part (ii) follows from (i), Proposition~\ref{prop:semistandardBasis} and the definition
of Schur functions (see~\eqref{eq:Schur} in~\S\ref{subsec:plethysm}).
\end{proof}

In the proofs of the main theorems, it is easiest to consider $\nabla^\lambda(E)$ as a 
module for the Lie algebra $\gl(E)$ of $\GL(E)$.
Let $\beta$ be a composition of $n$ with $\ell(\beta) = d$.
Recall that if $V$ is a $\gl(E)$-module then  $v \in V$ is a \emph{weight vector}
of \emph{weight} $\beta$ if
\[ X \cdot v = (\beta_1 X_{11} + \cdots + \beta_d X_{dd}) v \]
for all diagonal matrices $X \in \gl(E)$. If this equation holds for all diagonal $X \in \b$,
and $X \cdot v = 0$ whenever $X$ is strictly upper-triangular,
then we say that $v$ is a \emph{highest-weight vector} of weight $\beta$.
Let $\b$ be the Borel subalgebra of $\gl(E)$ of upper-triangular matrices.
For $c \in \{2,\ldots, d\}$, let 
$X^{(c)} \in \b$ be the strictly upper-triangular matrix having $1$ in position $(c-1,c)$
and $0$ in all other positions. Since the Lie subalgebra of $\b$ of
strictly upper-triangular matrices is generated by $X^{(2)},\ldots, X^{(c)}$, 
we have the following lemma.

\begin{lemma}\label{lemma:hw}
Let $V$ be a $\gl(E)$-module. The vector $v \in V$ is a highest-weight vector
of weight $\beta$ if and only if $X \cdot v = (\beta_1 X_{11} + \cdots + \beta_d X_{dd})v$
for~all diagonal matrices $X \in \gl(E)$ and $X^{(c)} \cdot v = 0$ for each $c \in \{2,\ldots, d\}$.
\hfill$\qed$
\end{lemma}

Using Lemma~\ref{lemma:hw} we establish the remaining basic properties of $\nabla^\lambda(E)$.
(Thus we take $V$ to be the natural representation $E$ of $\GL(E)$, and work in the basis $e_1, \ldots, e_d$.)
The main novel feature is the use of Proposition~\ref{prop:boundGL} to prove (v) and~(vi).

Recall that if $V$ and $W$ are $\gl(E)$-modules then the action of $\gl(E)$ on $V \otimes
W$ and $\Sym^r \! V$ is defined by linear extension of 
$x\cdot (v \otimes w) = (x \cdot v) \otimes w + v \otimes (x \cdot w)$ and
$x\cdot (v_{i_1} \ldots v_{i_r}) = (x \cdot v_{i_1}) v_{i_2} \ldots 
v_{i_r} + \cdots + v_{i_1} v_{i_2} \ldots v_{i_{r-1}}
(x \cdot v_{i_r})$.
The $\lambda$-tableau~$t^\lambda$ in (iv)
was defined in \S\ref{subsec:pts} by $t_{(i,j)} = i$ for each $(i,j) \in [\lambda]$.

\bigskip
\begin{proposition}\label{prop:nabla}
Let $s \in \SSYT_{\{1,\ldots,d\}}(\lambda)$ and let $c \in \{2,\ldots, d\}$.
\begin{thmlist} 
\item If $s$ has content $\beta$ then $F(s)$ is a weight vector of weight $\beta$ for the action of $\gl(E)$.
\item $X^{(c)} \cdot F(s) = \sum_{t} F(t)$ where the sum is over all tableaux $t$ obtained
from~$s$ by changing a single entry from $c$ to $c-1$.
\item If every  $c$ in $s$ has a $c-1$  immediately above
it then $X^{(c)} \cdot F(s) = 0$.
\item $F(t^\lambda)$ is a highest-weight vector of weight $\lambda$.
\item Suppose that $s$ has a  $c$ not having  a  $c-1$ immediately above it.
Find the highest row of $s$ containing such an entry, and let $t$ be the tableau
obtained by changing the leftmost $c$ in this row to $c-1$. Then $t$ is semistandard
and $X^{(c)} \cdot F(s) = \epsilon F(t) + y$ where $\epsilon \in \N$ and
$y$ is an integral linear combination of $F(u)$ for semistandard $\lambda$-tableaux $u$
such that $t \rhd u$.
\item If $K$ has characteristic zero and $v \in \nabla^\lambda(E)$ is such that
$X^{(c)} \cdot v = 0$ for all $c \in \{2,\ldots, d\}$ then $v$ is a multiple of $F(t^\lambda)$.
\item If $K = \C$ then $\nabla^\lambda(E)$ is irreducible.
\end{thmlist}
\end{proposition}

\begin{proof}
 Part (i) follows easily from the rules for the action of $\gl(E)$ 
and the definition of $f(t)$ as a tensor product of symmetric powers, 
in analogy with Lemma~\ref{lemma:nabla}(i).
Since $X^{(c)} \cdot v_c = v_{c-1}$ and $X^{(c)} \cdot v_b = 0$ if $b\not = c$, we also have
\[ X^{(c)} \cdot f(s) 
= X^{(c)} \cdot 
\bigotimes_{i=1}^{\ell(\lambda)} \prod_{j=1}^{\lambda_i} e_{s_{(i,j)}} 
= \sum_t \bigotimes_{i=1}^{\ell(\lambda)} \prod_{j=1}^{\lambda_i} e_{t_{(i,j)}}  
= \sum_t f(t). \] 
where the sums are over all tableaux $t$ obtained
from~$s$ by changing a single entry from $c$ to $c-1$.
This proves the analogue of (ii) for $\GL$-tabloids, and~(ii) now follows from
the definition of $F(s)$. 
By~\eqref{eq:column},
$F(t) = 0$ whenever~$t$ has a repeated entry in a column, so
(iii) follows from (ii). Now (iv) follows from the definition of $t^\lambda$ 
in~\S\ref{subsec:pts},~(i), (iii) and Lemma~\ref{lemma:hw}.

For (v), let row~$a$ be the row of $s$ containing the chosen entry 
$c$. 
By choice of~$a$, the tableau $t$ is semistandard. 
By (ii),
$X^{(c)} \cdot F(s) = F(t) + \sum_{t'} F(t')$
where the sum is over all tableaux $t'$ obtained from $s$ by changing a different $c$ to a $c-1$.
If this $c$ has a $c-1$ above it then $F(t') = 0$. If not, and this $c$ is in row $a$,
then by Proposition~\ref{prop:boundGL},
$F(t') = F(t) + z$ where $z$ is an integral linear combination of $F(u)$ for semistandard $\lambda$-tableaux
$u$ such that $t \rhd u$. In the remaining case $c$ is in a lower row than row $a$, and $t'$ is column
standard
with $t \rhd \overline{t'}$. Therefore, by Proposition~\ref{prop:boundGL},
\[ X^{(c)} \cdot F(s) = \epsilon F(t) + y \]
where $\epsilon$ is the number of entries $c$ in row $a$ of $t$ not having $c-1$ immediately
above them, and $y$ is as required. 

For (vi), let $v = \sum_{s \in \mathcal{S}} \gamma_s F(s)$ where $\mathcal{S} \subseteq \SSYT_{\{1,\ldots,d\}}(\lambda)$ and $\gamma_s \not= 0$ for each~$s \in \mathcal{S}$. 
Suppose that $v$ is not a multiple of 
$F(t^\lambda)$.
Choose $a$  minimal such that some
$s \in \mathcal{S}$ has an entry $c$ in row $a$ with $c > a$. Choose $c$ minimal with
this property and suppose that $s$ has $\epsilon$ such entries $c$, where $\epsilon$ is maximal. 
Thus rows $a-1$ and $a$
of~$s$ have the form
\begin{align*} 
& \begin{matrix} a^- & \ldots & a^- & a^- &  \ldots & a^-  & a^- & \ldots & a^- \\
                 a^{\phantom{-}}   & \ldots & a^{\phantom{-}}    & c^{\phantom{-}}    &   \ldots & c^{\phantom{-}}    & \ldots \end{matrix} \\[-15pt]
& \hskip0.925in              \underbrace{\hskip 0.72in}_{\epsilon}                 
\end{align*}
where $a^-$ denotes $a-1$. Replacing any $c$ in row $a$ of $s$ with $c-1$ gives a column-standard tableau;
moreover, replacing the leftmost such $c$ gives a semistandard $\lambda$-tableau. Let $t$ be this tableau.
Applying (v) to the summands of $X^{(c)} \cdot v$ we see that
\[ X^{(c)} \cdot F(s) = \epsilon 
F(t) + y \]
where $y$ is is an integral linear combination of $F(u)$ for semistandard $\lambda$-tableaux
such that $t \rhd u$. Now suppose that $\gamma_{s'} \not= 0$. Let $t'$ be 
obtained from $s'$ by changing a single $c$ to $c-1$; we may suppose this $c$ has no $c-1$ above it,
and so $t'$ is column-standard. Then, comparing $\overline{t'}^{\le c}$ and $t^{\le c}$ on row $a$, our choice 
of $s$ implies that  
$\overline{t'} \not\!\!\unrhd \, t$. Therefore, by Proposition~\ref{prop:boundGL},
the coefficient of $F(t)$ in $F(\overline{t'})$ is zero. Hence the coefficient of $F(t)$ in 
$X^{(c)} \cdot v$ 
is $\epsilon \gamma_s$. Since $K$ has
characteristic zero, this is a contradiction.

Finally, by Proposition 14.13 of \cite{FultonHarrisReps}, 
the $\gl(E)$-submodule $V$ of $\nabla^\lambda(E)$ generated by 
$F(t^\lambda)$ is irreducible. If $V$ is a proper submodule of $\nabla^\lambda(E)$ then, 
by Weyl's Theorem (see \cite[\S 6.3]{Humphreys} or \cite[Appendix~B]{ErdmannWildon}),
$V$ has
 a complementary submodule.
 By 
Proposition 14.13 of \cite{FultonHarrisReps} this complement
  contains a highest-weight vector, contradicting (vi).
Hence $\nabla^\lambda(E)$ is irreducible. 
\end{proof}

We end with a result summarizing the remaining properties we need of polynomial $\GL(E)$-modules.

\begin{proposition}\label{prop:hw}
Let $V$ be a polynomial $\GL(E)$-module of degree $r$.
\begin{thmlist}
\item $V$ contains a highest-weight vector.
\item If $v \in V$ is a highest-weight vector of weight $\lambda$
then $v$ generates a submodule of $V$ isomorphic to $\nabla^\lambda(E)$.
\item $V$ is isomorphic to a direct sum of certain $\nabla^\lambda(E)$ for $\lambda \in \Par(r)$.
\item Let $\Phi_V$ be the formal character of $V$. If $\lambda \in \Par(r)$ has at most $\dim E$ parts 
then $\langle \Phi_V, s_\lambda \rangle = [V : \nabla^\lambda(E)]$.
\end{thmlist}

\end{proposition}

\begin{proof}
Parts (i) and (ii) follow from 
Proposition 14.13 of \cite{FultonHarrisReps} and
Proposition \ref{prop:nabla}(vii); (iii) then follows from
 Weyl's Theorem 
(see \cite[\S 6.3]{Humphreys} or \cite[Appendix~B]{ErdmannWildon}). 
By (iii) it suffices to prove (iv) when $V$ is irreducible, in which case it is immediate
from the orthogonality of Schur functions and Lemma~\ref{lemma:nabla}.
\end{proof}

\begin{remark}\label{remark:James}
Our definition of $\GL$-tabloids and $\GL$-polytabloids is in deliberate
analogy with the definitions of tabloids and polytabloids in \cite[(3.9), (4.3)]{James}.
In \cite[Ch.~26]{James}, James uses his results on symmetric groups to
construct a polynomial representation $W^\lambda$ of $\GL(E)$, defined over an arbitrary field $K$.

Let
 $\mathcal{G}_\otimes: \SymG_r\text{-mod} \rightarrow \GLEmod_r$ be the inverse Schur functor
 sending
 a $\SymG_r$-module $U$ to $E^{\otimes r} \otimes_{K\SymG_r} U$. (Here $E^{\otimes r}$ 
 is regarded as a right $KS_r$-module by place permutation on tensors.)
James' module $W^\lambda$
may be defined as the image of $\mathcal{G}_{\otimes}(S^\lambda)$ under the canonical
map $\mathcal{G}_\otimes S^\lambda \rightarrow \mathcal{G}_\otimes M^\lambda$
induced by the inclusion $S^\lambda \rightarrow M^\lambda$ of the Specht module~$S^\lambda$
into the Young permutation module~$M^\lambda$. It follows from the definition
of $\GL$-polytabloids and the isomorphism
$\mathcal{G}_\otimes(M^\lambda) \cong
\bigotimes_{i=1}^{\ell(\lambda)} \Sym^{\lambda_i}\!E$ that
$W^\lambda \cong \nabla^\lambda(E)$. 
By working throughout in $\bigotimes_{i=1}^{\ell(\lambda)} \Sym^{\lambda_i}\!E$ we avoid
the nasty technicality that, in general, $\mathcal{G}_\otimes(S^\lambda) \not\cong W^\lambda$.
For example, since $S^{(2)} \cong S^{(1,1)} \cong M^{(2)} \cong K$ when~$K$ has characteristic $2$, 
and $\mathcal{G}_\otimes(K) = (E^{\otimes 2})_{\otimes KS_2} K \cong \Sym^2 E$, we have
$\mathcal{G}_\otimes (S^{(2)}) \cong \mathcal{G}_\otimes (S^{(1,1)}) \cong \Sym^2 E \cong W^{(2)}$.
But $W^{(2)} \not\cong W^{(1,1)} \cong \bigwedge^2 E$.

For more recent work on inverse Schur functions we refer the reader to~\cite{HemmerNakano} and~\cite{McDowellInverseSchur}.
A related homological remark, which explains our notation, is that $\nabla^\lambda(E)$ is a co-standard
module in the sense of quasi-hereditary algebras: see \cite{DlabRingel}.
\end{remark}

\setcounter{section}{2}


\section{A model for $\nablanumuE$}
\label{sec:model}

Fix $\mu \in \Par(m)$ and $\nu \in \Par(n)$. 
Let $E$ be a complex vector space with basis $\u_1, \ldots, \u_d$.
Throughout this section let
$\B = \SSYT_{\{1,\ldots, d\}}(\mu)$ ordered by the total order $<$ in Definition~\ref{defn:totalorder}.

\begin{lemma}\label{lemma:numubasis}
The $\GL(E)$-module $\nabla^\nu\bigl( \nabla^\mu(E) \bigr)$ has as a canonical basis
the~set
\[ \bigl\{ F(S) : S \in \SSYT_\B(\nu) \bigr\}. \]
of plethystic semistandard tableaux of shape $\mu^\nu$ whose $\mu$-tableau entries
each have entries from $\{1,\ldots, d\}$.
\end{lemma}

\begin{proof}
By Proposition~\ref{prop:semistandardBasis}, applied with $V = E$ and $\mathcal{B} = \{1,\ldots, d\}$, 
$\nabla^\mu(E)$ has
$
\{F(t) : t \in \SSYT_{\{1,\ldots, d\}}(\mu) \}$ as a basis. 
The lemma now follows from
another application of Proposition~\ref{prop:semistandardBasis}, this time with $V = \nabla^\mu(E)$
and $\B = \SSYT_{\{1,\ldots, d\}}(\mu)$.
\end{proof}

As a notational guide, 
we use upper case letters to denote $\nu$-tableaux whose entries are
$\mu$-tableaux and upper case indices $I$ and $J$ to refer to their rows and columns.
 
By Definition~\ref{defn:polytabloids}, 
\begin{equation}\label{eq:polyTabloid1} F(S) 
= \!\!\!\sum_{\tau \in \CPP(\nu)} f(S \tau)  \sgn(\sigma). \end{equation}
Since a $\mu$-tableau entry $s = S_{\pairpq\tau^{-1}} \in \B$ corresponds to the 
basis vector $F(s)$ of $\nabla^\mu(E)$, we have
\[ f(S \tau) = \bigotimes_{I=1}^{\ell(\nu)} \prod_{J=1}^{\nu_I} 
F(S_{\pairpq\tau^{-1}})\sgn(\tau). \]
In turn, 
\[ F(S_{\pairpq\tau^{-1}} ) = \!\!\!\sum_{\sigma \in \CPP(\mu)} f(s  \sigma)\sgn(\sigma) 
=\!\!\! \sum_{\sigma \in \CPP(\mu)} \bigotimes_{i=1}^{\ell(\mu)} \prod_{j=1}^{\mu_i}
u_{s_{(i,j)\sigma^{-1}}} \sgn(\sigma). \]
Thus $F(S) \in \bigotimes_{I=1}^{\ell(\nu)} \Sym^{\nu_I}
\bigl( \bigotimes_{J=1}^{\ell(\mu)} \Sym^{\mu_J}(E) \bigr)$.
It will be convenient to define the \emph{weight} of a tableau $S \in \SSYT_\B(\nu)$,
denoted $\wt(S)$, to be the sum of the contents of its $\mu$-tableau entries.

For example, take $\B = \SSYT_{\{1,2,3\}}\bigl( (2,1) \bigr)$. If $\nu = (2,2)$ and $\mu = (2,1)$ and 
\[  \setlength{\arrayrulewidth}{.06em}
\makeatletter
\def\y@vr{\vrule height0.8\y@b@xdim width\y@linethick depth 0.3\y@b@xdim}
\def\y@setdim{%
  \ify@autoscale%
   \ifvoid1\else\typeout{Package youngtab: box1 not free! Expect an
     error!}\fi%
   \setbox1=\hbox{A}\y@b@xdim=1.8\ht1 \setbox1=\hbox{}\box1%
  \else\y@b@xdim=\y@boxdim \advance\y@b@xdim by -2\y@linethick
  \fi}
\makeatother
S = \begin{tabular}{|c|c|}  \hline
$\young(11,2)$\rule[-18pt]{0pt}{42pt} & $\young(11,2)$ \\
\cline{1-2}
$\young(11,3)$\rule[-18pt]{0pt}{42pt} & $\young(12,2)$ \\
\cline{1-2}
\end{tabular} \in \SSYT_\B\bigl( (2,2) \bigr)
\]
 then $\wt(T) = (7,4,1)$ and
\begin{align*}
 F(S) &=\, \scriptstyle F\bigl( \, \syoung{(11,2)}\, \bigr)F\bigl(\, \syoung{(11,2)}\, \bigr) 
\otimes F\bigl(\, \syoung{(11,3)}\, \bigr) F\bigl(\, \syoung{(12,2)}\, \bigr)
- F\bigl(\, \syoung{(11,3)}\, \bigr)F\bigl(\, \syoung{(11,2)}\, \bigr) 
\otimes F\bigl(\, \syoung{(11,2)}\, \bigr)F\bigl(\, \syoung{(12,2)}\, \bigr) \\ 
&\ \ \scriptstyle -\, F\bigl(\, \syoung{(11,2)}\, \bigr)F\bigl(\, \syoung{(12,2)}\, \bigr) 
\otimes F\bigl(\, \syoung{(11,3)}\, \bigr)F\bigl(\, \syoung{(11,2)}\, \bigr)  
+ F\bigl(\, \syoung{(11,3)}\, \bigr)F\bigl(\, \syoung{(12,2)}\, \bigr) \otimes
F\bigl(\, \syoung{(11,2)}\, \bigr)F\bigl(\, \syoung{(11,2)}\, \bigr) . \end{align*}
where $F\bigl(\, \young(11,2) \, \bigr) = \u_1^2\otimes \u_2 - \u_2\u_1 \otimes \u_1$, and so on.

\begin{proposition}\label{prop:nablaT}
Let $S \in \SSYT_\B(\nu)$.
\begin{thmlist}
\item If $\wt(S) = \beta$ then $F(S) \in \nabla^\nu\bigl( \nabla^\mu(E) \bigr)$ 
is a weight vector of weight $\beta$.
\item $X^{(c)} \cdot F(S) = \sum_{T} F(T)$ where the sum is over all $\nu$-tableaux $T$
obtained from~$S$ by changing a single $c$ to $c-1$ in a single $\mu$-tableau entry.
\end{thmlist}
\end{proposition}

\begin{proof}
Apply Proposition~\ref{prop:nabla}(i) and (ii) to~\eqref{eq:polyTabloid1}.
\end{proof}

In particular, the canonical basis defined in~Lemma~\ref{lemma:numubasis} 
for $\nablanumuE$ consists of weight 
vectors.

\begin{proposition}\label{prop:compositionPlethysm}
The formal character of $\nabla^\nu 
\bigl(\nabla^\mu(E)\bigr)$ is $(s_\nu \circ s_\mu)(x_1,\ldots, x_d)$.
\end{proposition}

\begin{proof}
By the definition of Schur functions
in~\eqref{eq:Schur} and the definition of plethysm given shortly afterwards,
$(s_\nu \circ s_\mu)(x_1,\ldots,x_d)$ is obtained by evaluating $s_\nu$ at the monomials $x^t$ for $t \in \B$.
Thus 
\[ (s_\nu \circ s_\mu)(x_1,\ldots,x_d) = \sum_{S \in \SSYT_\B(\nu)} x^{\wt(S)}. \]
It follows that if $\beta$ is a composition of $mn$ with $\ell(\beta) = d$
then the coefficient $x^\beta$ in $s_\nu \circ s_\mu$
is the number of $S \in \SSYT_\B(\nu)$ of weight $\beta$. By Lemma~\eqref{lemma:numubasis} and Proposition~\ref{prop:nablaT}(i),
this is the dimension of the $\beta$-weight space in $\nabla^\nu\bigl(\nabla^\mu(E)\bigr)$.
\end{proof}

\section{Proof of Theorem~\ref{thm:BCV-N}}\label{sec:BCV-N}
We use the model for $\nablanumuE$ in 
\S\ref{sec:model}, taking $\dim E = d$ where $d \ge mn$. 
It will be convenient to number the rows of $[(r) \sqcup \mu]$ from $0$,
so that 
\[ [(r) \sqcup \mu] = \{(0,j) : 1 \le j \le r\} \cup [\mu]. \]
Given a
$\mu$-tableau $t$ with entries from $\{1,\ldots, d\}$, let $\widetilde{t}$ be
the $(r) \sqcup \mu$-tableau with entries from $\{1, \ldots, d, d+1\}$ 
defined by
\[ \widetilde{t}_{(i,j)} = \begin{cases} t_{(i,j)} + 1 & \text{if $i \ge 1$} \\
1 & \text{if $i=0$.} \end{cases} \]
Thus $\widetilde{t}$ is obtained from $t$ by increasing each
entry by $1$ and then inserting a new row of $1$s of length $r$ at the top.

The following technical lemma shows that 
each snake relation satisfied by $F(t)$ gives a very similar relation
satisfied by $\Ftt$. We use this to show
in the proof of 
Proposition~\ref{prop:BCV-Nbijection} that $\Ftt$ can be straightened
in essentially the same way as $F(t)$.

\begin{lemma}\label{lemma:technical}
Let $t$ be a $\mu$-tableau with entries from $\{1,\ldots, d\}$.
Let $(i,j) \in [\mu]$ with $j < \mu_1$. 
Let $F(t) = -\sum_{c=2}^\ell F(t \phi_c) \sgn(\phi_c)$
be a snake relation as in~\eqref{eq:snake}, with $j' = j+1$.
Then
\[ \Ftt = -\sum_{c=2}^\ell  F(\widetilde{\,t\phi_c\,} ) \sgn(\phi_c). \]
\end{lemma}

\begin{proof}
Let $A = A_{\mu}(i,j)$ and $B = B_{\mu}(i,j+1)$ be as in~\eqref{eq:snake}.
Let $B^+ = B \cup \{(0,j+1)\}$.
(Recall that the rows of $[(r) \sqcup \mu])$ are numbered from~$0$.)
Let $C^+ = A \cup B^+$.
By hypothesis, $\phi_2, \ldots, \phi_\ell$ are representatives
for the proper left cosets of $S_A \times S_B$ in $S_{A \cup B}$. 
The permutations in $\phi_c(S_A \times S_{B^+})$ fixing $(0,j+1)$ 
are precisely the elements of $\phi_c(S_A \times S_{B})$. Therefore
the cosets $\phi_c(S_{A} \times S_{B^+})$
for $2 \le c \le \ell$ are disjoint.
 Let $\phi_1 = \id$ and choose further coset representatives
$\phi_{\ell+1}, \ldots, \phi_{\ell^+}$ such that
$\phi_1, \ldots, \phi_{\ell^+}$ is a full set of representatives for the left
cosets of $S_{A} \times S_{B^+}$ in $S_{C^+}$.
By~\eqref{eq:snake} we have
\begin{equation}\label{eq:technical}
\Ftt = -\sum_{c=2}^{\ell^+} \Fttphic \sgn(\phi_c).
\end{equation}
Suppose that $c > \ell$. Then $(0,j+1)\phi_c \in A$, and 
since $\widetilde{t}_{(0,j)} = 1$ and $\widetilde{t}_{(0,j+1)} = 1$,
it follows that
$\widetilde{t}\hskip1pt \phi_c$ has two entries equal to $1$ in column $j$. Therefore
$\Fttphic  = 0$ by~\eqref{eq:column}.
We may therefore replace the upper limit in the sum in~\eqref{eq:technical} with $\ell$.
After making this change,~\eqref{eq:technical} is precisely the relation we require.
\end{proof}

Recall from~\S\ref{sec:model} that $\B = \SSYT_{\{1,\ldots,d\}}(\mu)$
is ordered by the total order $<$ in Definition~\ref{defn:totalorder}.
Let $\Bp = \SSYT_{\{1,\ldots,d,d+1\}}((r) \sqcup \mu)$. 
Given $S \in \SSYT_\B(\nu)$, let~$\widetilde{S}$ be the $\nu$-tableau
defined by replacing each $\mu$-tableau entry $s$ of $S$ with~$\widetilde{s}$.
For $s$, $t \in \B$ we have
$s < t$ if and only if $\widetilde{s} < \widetilde{t}$. Hence
$\widetilde{S} \in \SSYT_\Bp(\nu)$.


Let $E^+ = E \oplus \langle \u_{d+1} \rangle$ be a $(d+1)$-dimensional complex vector space.
Recall that~$V_\lambda$ denotes the $\lambda$-weight space of a $\gl(E)$-module $V$.

\begin{lemma}\label{lemma:tildeiso}
The map $F(S) \mapsto F(\widetilde{S})$ defines a $\C$-linear isomorphism 
\[ \nablanumuE_\lambda \rightarrow \nablanumuEt_{(nr)\sqcup
\lambda}. \]
\end{lemma}


\begin{proof}
Suppose that $v \in  \nablanumuEt$ is a weight 
vector of weight $(nr) \sqcup \lambda$. 
Let $T \in \SSYT_{\Bp}(\nu)$ and 
suppose that the coefficient of $F(T)$ in $v$ is non-zero.
By Proposition~\ref{prop:nablaT}(i), there are $nr$ entries equal to $1$ in the $\mu$-tableau entries 
of $T$.
Since $(r)$ is the largest
part of $((r) \sqcup \mu)$, each $T_{(g,h)}$ for $(g,h) \in [\nu]$ 
has at most $r$ entries equal to $1$. Therefore each $T_{(g,h)}$ has exactly $r$ entries 
equal to $1$, necessarily
lying in its longest row. Hence $T = \widetilde{S}$ for a unique $S \in \SSYT_\B(\nu)$,
and so the map is surjective. Since it sends basis elements to basis elements, it is injective.
\end{proof}

Let $\widetilde{v} \in \nablanumuEt_{(nr) \sqcup \lambda}$ 
denote the image of $v \in \nablanumuE_\lambda$ under the map in the previous lemma.

%

\begin{proposition}\label{prop:BCV-Nbijection}
The map $v \mapsto \widetilde{v}$ restricts to a bijection between highest-weight 
vectors in $\nablanumuE$ of weight $\lambda$ and 
highest-weight
vectors in $\nablanumuEt$ of weight $(nr) \sqcup \lambda$.
\end{proposition}

\begin{proof}
Let 
\[ v = \sum_{S \in \SSYT_\B(\nu)} \alpha_S F(S) \in \nablanumuE_\lambda .\]
By definition, 
\[ \widetilde{v} = \sum_{S \in \SSYT_{\B}(\nu)} \alpha_S F(\widetilde{S}) 
\in \nablanumuEt_{\lambda + (nr)}. \]
Let $S \in \SSYT_\B(\nu)$.
 Since 
changing any $2$ to $1$ in a $((r) \sqcup \mu)$-tableau entry $\widetilde{s}$ of $\widetilde{S}$
gives a $((r) \sqcup \mu)$-tableau with two $1$s in the same column, Proposition~\ref{prop:nablaT}(ii) implies
that $X^{(2)} \cdot F(\widetilde{S}) = 0$.
Now let $c \in \{2,\ldots,d\}$. 
Again by Proposition~\ref{prop:nablaT}(ii), 
$X^{(c)} \cdot F(S) = \sum_{T} F(T)$ where the sum is over all $T \in \SSYT_\mathcal{B}(\nu)$
obtained from $S$ by changing a single $c$ to $c-1$. Moreover,
$X^{(c+1)} \cdot F(\widetilde{S}) = \sum_{T} F(\widetilde{T})$ \emph{with the
same conditions on the sum}.

Suppose that
$v$
is a highest-weight vector.
By the previous paragraph and~\eqref{eq:polyTabloid1}, each summand
$F(T)$ appearing
in $X^{(c)} \cdot F(S)$ (respectively, each $F(\widetilde{T})$ appearing
in $X^{(c+1)} \cdot F(\widetilde{S})$)
is a sum of tensor products
of symmetric products of $F(u)$ (respectively $F(\widetilde{u})$)
for certain $\mu$-tableaux $u$ (respectively $((r) \sqcup \mu)$-tableaux $\widetilde{u}$),
at most one of which, say $t$ (respectively $\widetilde{t}\hskip0.5pt$), 
is non-semistandard. By Corollary~\ref{cor:snake}, we may straighten $F(t)$
to a linear combination of $F(s)$ for $s \in \SSYT_{\{1,\ldots, d\}}(\mu)$ 
by a sequence of snake relations~\eqref{eq:snake} swapping boxes between adjacent
columns. 
By multilinearity,
this expresses $F(T)$ as a linear combination of $F(S)$ for $S \in \SSYT_\B(\nu)$.
Recall that the rows of $[(r) \sqcup \mu]$ are labelled from $0$.
By Lemma~\ref{lemma:technical} if we apply \emph{exactly the same sequence of relations}
to straighten $F(\widetilde{t})$, we express $F(\widetilde{T})$ 
as a linear combination of $F(\widetilde{S})$ for $S \in \SSYT_\B(\nu)$
\emph{with the same coefficients}. 
Hence $X^{(c)} \cdot v = 0$ implies $X^{(c+1)} \cdot \widetilde{v} = 0$.
By Lemma~\ref{lemma:hw}, $\widetilde{v}$ is a highest-weight vector.

Conversely, if $\widetilde{v}$
is a highest-weight vector then running this argument in reverse
shows that $X^{(c)} \cdot v = 0$ for $c \in \{2,\ldots, d\}$, and so $v$ is a highest-weight vector.
\end{proof}

By Proposition~\ref{prop:BCV-Nbijection} and
Proposition~\ref{prop:hw} we have
\[ [\nabla^\nu (\nabla^\mu E) : \nabla^\lambda E] = [\nabla^\nu (\nabla^{(r) \sqcup \mu})(E) :
\nabla^{(nr) \sqcup \lambda}(E)].\]
Theorem~\ref{thm:BCV-N} now follows using~\eqref{eq:schurnabla}.

\begin{remark}\label{remark:BCVnz}
We remark that since any plethystic semistandard tableau of shape~$\mu^\nu$
has at most~$n \mu_1$ integer entries of $1$, if $\lambda_1 > n \mu_1$
then both sides of the equation in Theorem~\ref{thm:BCV-N} are zero.
\end{remark}

\begin{remark}\label{remark:KahleMichalek}
In \cite[Lemma 3.2]{KahleMichalek} a proof of the special case $\mu = (1^m)$ and $r=1$ is indicated.
In our notation, the authors
consider $\nabla^\nu \bigl( \bigwedge^{m+1}(E) \bigr)$ as a submodule of $\bigl(\bigwedge^{m+1}(E)\bigr)^{\otimes n}$
and observe that each tensor summand in a highest-weight vector $v$ of weight $(n) \sqcup
\lambda$ is of the
form $(e_1 \wedge \cdots ) \otimes \cdots \otimes (e_1 \wedge \cdots )$. They define
a map into $\bigl(\bigwedge^m (E)\bigr)^{\otimes n}$ by removing $e_1$ from each tensor factor of $v$
and reducing the indices. This is essentially the inverse map to ours, in this special case. 
\end{remark}


\section{Proof of Theorem~\ref{thm:Brion}}\label{sec:Brion}

We adapt the strategy used to prove Theorem~\ref{thm:BCV-N}, again
working in the model $\nablanumuE$ from \S 3, now taking $\dim E = d$ where
$d \ge r$ and $d \ge \ell(\mu)$.

If $r \ge \ell(\mu)$ then set $\e=1$. Otherwise let $\e = \mu_{r+1}+1$.
To relate 
$[\mu]$
and $[\mu + (1^r)]$ we use the following notation.
Let $\altbox{i,0} = (i,\e)$ for $1 \le i \le r$ and for $(i,j) \in [\mu]$, let
\[  \altbox{i,j} = \begin{cases} (i,j) & \text{if $j < \e$} \\
(i,j+1) & \text{if $j \ge \e$.} \end{cases} \] 
As illustrated in~Example~\ref{ex:insertion}, we have
\[ [\mu + (1^r)] = \bigl\{ 
\altbox{i,0} : 1 \le i \le r \bigr\} \medcup \bigl\{ \altbox{i,j} : (i,j) \in [\mu] \bigr\}. \]
Given a $\mu$-tableau $t$ with entries from $\{1,\dots,d\}$, 
let $\bt{t}$ be the $(\mu + (1^r))$-tableau defined by
$\bt{t}_{[i,0]} = i$ for $1 \le i \le r$
and $\bt{t}_{[i,j]} = t_{(i,j)}$ if $j > 0$.
Thus 
$\bt{t}$ is  obtained from~$t$ by  inserting a new column $e$ with entries $1$, \ldots, $r$,
moving the existing column $e$ and other later numbered columns one position right. 

\begin{example}\label{ex:insertion}
If $r = 2$ and $\mu = (4,2,1)$ then $e= 2$;
the labels for the boxes in $[(4,2,1)]$ and $[(4,2,1) + (1,1)]$ are

\smallskip
\begin{center}
\begin{tikzpicture}[x=0.75cm,y=-0.75cm]
\draw (0,0)--(4,0); 
\draw (0,1)--(4,1);
\draw (0,2)--(2,2);
\draw (0,3)--(1,3);
\draw (0,0)--(0,3);
\draw (1,0)--(1,3);
\draw (2,0)--(2,2);
\draw (3,0)--(3,1);
\draw (4,0)--(4,1);
\node at (0.5,0.5) {$\scriptstyle (1,1)$};
\node at (1.5,0.5) {$\scriptstyle (1,2)$};
\node at (2.5,0.5) {$\scriptstyle (1,3)$};
\node at (3.5,0.5) {$\scriptstyle (1,4)$};
\node at (0.5,1.5) {$\scriptstyle (2,1)$};
\node at (1.5,1.5) {$\scriptstyle (2,2)$};
\node at (0.5,2.5) {$\scriptstyle (3,1)$};
\end{tikzpicture}\ \raisebox{36pt}{,}
\qquad
\begin{tikzpicture}[x=0.75cm,y=-0.75cm]
\draw (0,0)--(5,0); 
\draw (0,1)--(5,1);
\draw (0,2)--(3,2);
\draw (0,3)--(1,3);
\draw (0,0)--(0,3);
\draw (1,0)--(1,3);
\draw (2,0)--(2,2);
\draw (3,0)--(3,2);
\draw (4,0)--(4,1);
\draw (5,0)--(5,1);
\node at (0.5,0.5) {$\scriptstyle [1,1]$};
\node at (1.5,0.5) {$\scriptstyle [1,0]$};
\node at (2.5,0.5) {$\scriptstyle [1,2]$};
\node at (3.5,0.5) {$\scriptstyle [1,3]$};
\node at (4.5,0.5) {$\scriptstyle [1,4]$};
\node at (0.5,1.5) {$\scriptstyle [2,1]$};
\node at (1.5,1.5) {$\scriptstyle [2,0]$};
\node at (2.5,1.5) {$\scriptstyle [2,2]$};
\node at (0.5,2.5) {$\scriptstyle [3,1]$};
\end{tikzpicture}
\raisebox{36pt}{.}
\end{center}
Two pairs of a $\mu$-tableau $t$ and the corresponding $(\mu + (1^r))$-tableau $t^\star$ are 
\[ \left(\, \young(1123,23,4),\ \young(11123,223,4)\, \right),\quad
\left(\, \young(1123,33,4),\ \young(11123,323,4)\, \right). \]
\end{example}


Given $\phi \in S_{[\mu]}$, let $\bt{\phi} \in S_{[\mu] + (1^r)}$ be defined by
$[i,0]\bt{\phi} = [i,0]$ for $1 \le i \le r$ and 
$[i,j]\bt{\phi} = [i',j'] \iff (i,j)\phi = (i',j')$.
In analogy with Lemma~\ref{lemma:technical}, we now
show that $F(\bt{t})$ satisfies the appropriate conjugate of each snake relation 
(see~\eqref{eq:snake} after Lemma~\ref{lemma:GLgarnir}) satisfied
by $F(t)$.

\begin{lemma}\label{lemma:technicalhat}
Let $t$ be a $\mu$-tableau with entries from $\{1,\ldots, d\}$.
Let $(i,j) \in [\mu]$ with $j < \mu_1$. 
Let $F(t) = -\sum_{c=2}^\ell F(t \phi_c) \sgn(\phi_c)$ be a snake
relation as in~\eqref{eq:snake}, with $j' = j+1$. Then
\[ F(\bt{t}) = -\sum_{c=2}^\ell F(\bt{t}\hskip0.5pt\bt{\phi_c}) \sgn(\bt{\phi_c}). \]
\end{lemma}

\begin{proof}
The claimed relation is an instance of~\eqref{eq:snake} for $\bt{t}$, with
respect to the boxes $[i,j]$ and $[i,j+1] \in [\mu] + (1^r)$. 
\end{proof}

\begin{remark}
If $\e=1$ the added column in $\bt{t}$ is at the far left with entries $1,\ldots, r$, and Lemma~\ref{lemma:technicalhat}
may be compared with Lemma~\ref{lemma:technical} in which we add a row at the top with entries all equal to $1$:
an intuitive statement of 
Lemmas~\ref{lemma:technical} and~\ref{lemma:technicalhat} is that these additions
preserve snake relations. 
If instead $e > 1$ and
$j=\e-1$ then $[i,j]$ and $[i,j+1]$ lie in the non-adjacent
columns $\e-1$ and $\e+1$ of $[\mu] + (1^r)$;
this is the only case where we need the freedom in~\eqref{eq:snake} to take $j' \not= j+1$.
\end{remark}

Recall from~\S\ref{sec:model} that $\B = \SSYT_{\{1,\ldots,d\}}(\mu)$.
Let $\CC$ be the set of column-standard $(\mu + (1^r))$-tableaux with entries from $\{1,\ldots, d\}$
and let $\BB = \SSYT_{\{1,\ldots,d\}}(\mu+(1^r))$. Thus $\BB \subseteq \CC$.
Both $\B$ and $\CC$ are ordered by the total order $<$ in Definition~\ref{defn:totalorder}.
Given $T \in \SSYT_\B(\nu)$, let $\bt{T}$ be the $\nu$-tableau
defined by replacing each $\mu$-tableau entry $s$ of $T$ with~$\bt{s}$.
For $s$, $t \in \B$ we have
$s < t$ if and only if $\bt{s} < \bt{t}$,
since the inserted column is the same in $\bt{s}$ and~$\bt{t}$.
Hence $\bt{T} \in \SSYT_\CC(\nu)$
and $F(\bt{T}) \in \nabla^\nu \bigl( \nabla^{\mu+(1^r)}(E) \bigr)$.
 By the case $N=1$ of the following definition, $\bt{T} \in \SSYT_{\BB}(\nu)$
if and only if $\bt{T}$ is $r$-saturated.

\begin{definition}
Recall that if $r \ge \ell(\mu)$ then  $\e=1$ and otherwise $\e = \mu_{r+1}+1$.
Let $N \in \N_0$ and
let $U$ be a $\nu$-tableau whose entries are certain $\mu + N(1^r)$-tableaux.
We say that $U$ is \emph{$r$-saturated} if 
whenever $u$ is a $\mu + N(1^r)$-tableau entry of $U$, we have
$u_{(i,j)} = i$
for $1 \le i \le r$ and $1 \le j \le e$.
\end{definition}

Equivalently, $U$ is $r$-saturated if the first $e$ columns of 
each $\mu + N(1^r)$-tableau entry of $U$ each begin
$1, \ldots, r$ when read from top to bottom. 
For example, when $r=2$ and $\mu=(4,2,1)$ we saw
in Example~\ref{ex:insertion} that $e=2$. 
Taking $N=1$, of the two $((4,2,1)+(1^2))$-tableaux $t^\star$ shown,
only the first could be an
entry of a $2$-saturated 
tableau, since for the second $t^\star_{(2,1)} = 3$.

Given $v \in \nabla^\nu \bigl( \nabla^\mu (E) \bigr)$ 
let $\bt{v} \in \nabla^\nu \bigl( \nabla^{\mu+(1^r)} \bigr)$ denote the image of 
$v$ under the $\C$-linear map defined on the canonical basis in Lemma~\ref{lemma:numubasis} of
$\nabla^\nu \bigl( \nabla^\mu (E) \bigr)$ by $F(T) \mapsto F(T^\star)$
for each $T \in \SSYT_\B(\nu)$.

\begin{example}\label{ex:Brion}
Let $r = 2$ and let $\mu = (2,1,1)$ so $e=2$.
Let $d = 4$, so $\B = \SSYT_{\{1,2,3,4\}}\bigl((2,1,1)\bigr)$. Take $\nu = (2)$.
A calculation, either in {\sc Magma} or by hand
using the domino tableau rule in \cite[Theorem~4.1]{CarreLeclerc},
gives $\langle s_{(2)} \circ s_{(2,1,1)}, s_{(3,2,2,1)} \rangle = 1$
and so the space of highest-weight vectors of weight $(3,2,2,1)$ 
in $\Sym^2 \bigl( \nabla^{(2,1,1)}(E) \bigr)$ is $1$-dimensional. Computing
the images of the $F(s)$ for $s \in \B$ under the generators $X^{(2)}$, $X^{(3)}$, $X^{(4)}$ of the
Borel subalgebra $\b$ of upper-triangular matrices in $\gl_4(\C)$ 
using Proposition~\ref{prop:nabla}(ii) one finds that if
$T_{(1)}$, $T_{(2)}$, $T_{(3)}$, $T_{(4)}$ are the four tableaux
in $\SSYT_\mathcal{B}\bigl( (2) \bigr)$ shown below
\[ \doubletableau{1.2cm}{0.425cm}{\young(11,2,3)}{\young(12,3,4)}\; \raisebox{0.85cm}{,}\ 
  \doubletableau{1.2cm}{0.425cm}{\young(11,2,3)}{\young(13,2,4)}\; \raisebox{0.85cm}{,}\  
  \doubletableau{1.2cm}{0.425cm}{\young(11,3,4)}{\young(12,2,3)}\; \raisebox{0.85cm}{,}\ 
  \doubletableau{1.2cm}{0.425cm}{\young(11,2,4)}{\young(13,2,3)} \]
then, by Lemma~\ref{lemma:hw},
\[ v = F(T_{(1)}) - F(T_{(2)}) - F(T_{(3)}) + F(T_{(4)}) \in \Sym^2 \bigl( \nabla^{(2,1,1)}(E) \bigr). \]
is a highest-weight vector of weight $(3,2,2,1)$.
For example, by Proposition~\ref{prop:nabla}(ii), we have $X^{(2)} \cdot F(T^{(2)}) = 0$, $X^{(2)} \cdot F(T^{(4)}) = 0$ and
\[ X^{(2)} \cdot F(T^{(1)}) = 
F\left(\, \young(11,2,3)\, \right) F\left(\, \young(11,3,4)\, \right)
 = X^{(2)} \cdot F(T_{(3)}); \]
since $T_{(1)}$ and $T_{(3)}$ appear with opposite signs in $v$, this implies $X^{(2)} \cdot v = 0$. 
Essentially the same calculation shows that
\[ \bt{v} = F(\bt{T_{(1)}}) - F(\bt{T_{(2)}}) - F(\bt{T_{(3)}}) + F(\bt{T_{(4)}}) \in 
\Sym^2 \bigl( \nabla^{(2,1,1)+(1,1)}(E) \bigr)\]
is a highest-weight vector of weight $(3,2,2,1) + 2(1,1)$. 
The tableaux $\bt{T_{(1)}}$ and $\bt{T_{(3)}}$ are not $2$-saturated;
they lie in $\SSYT_{\CC}\bigl( (2) \bigr) $ but not in $\SSYT_{\BB}\bigl( (2) \bigr)$. For example
\begin{align*}
F( \bt{T_{(1)}} ) &= F\left(\; \young(111,22,3) \, \right) F\left(\; \young(112,32,4) \, \right) \\
&= F\left(\; \young(111,22,3) \, \right) F\left(\; \young(112,23,4) \, \right) -
 F\left(\; \young(111,22,3) \, \right) F\left(\; \young(112,24,3) \, \right) \end{align*}
expressed in the canonical basis of $\Sym^2 \bigl( \nabla^{(3,2,1)} (E) \bigr)$. We leave it as an exercise
to show that if $U_{(1)}$, $U_{(2)}$, $U_{(3)}$, $U_{(4)}$ are the four tableaux
in $\SSYT_\CC\bigl( (2) \bigr)$ shown below
\[ \scalebox{0.85}{$\displaystyle 
  \doubletableau{1.7cm}{0.425cm}{\young(111,23,3)}{\young(112,22,4)}\, \raisebox{0.85cm}{,}\ 
  \doubletableau{1.7cm}{0.425cm}{\young(111,22,4)}{\young(112,23,3)}\, \raisebox{0.85cm}{,}\  
  \doubletableau{1.7cm}{0.425cm}{\young(111,24,3)}{\young(112,22,3)}\, \raisebox{0.85cm}{,}\ 
  \doubletableau{1.7cm}{0.425cm}{\young(111,22,3)}{\young(112,24,3)}$} \]
then 
\[ w = F(U_{(1)}) - F(U_{(2)}) - F(U_{(3)}) + F(U_{(4)}) - F(T_{(2)}^\star) + 
F(T_{(4)}^\star)\]
is a highest-weight vector in $\Sym^2 \bigl( \nabla^{(2,1,1)+(1,1)}(E) \bigr)$
of weight $(3,2,2,1) + 2(1,1)$, linearly independent of $\bt{v}$. 
By Theorem~\ref{thm:Brion},
the multiplicity \[ [\Sym^2 \bigl( \nabla^{(2,1,1)+N(1,1)}(E) \bigr) : \nabla^{(3,2,2,1) + N(2,2)}] \] 
is constant
for $N \ge 1$. 
A further domino tableau calculation 
shows that $\langle s_{(2)} \circ s_{(3,2,1)}, s_{(5,4,2,1)} \rangle = 2$.
Therefore, repeating the column addition once more,
we obtain vectors \smash{$v^{\star\star}$} and \smash{$w^\star$} spanning the subspace
of $\Sym^2 \bigl( \nabla^{(2,1,1)+2(1,1)}(E) \bigr)$ of highest-weight vectors of weight $(3,2,2,1) + 2(2,2)$;
further additions give a spanning set for the subspace 
of  $\Sym^2 \bigl( \nabla^{(2,1,1)+N(1,1)}(E) \bigr)$ of highest-weight vectors of weight $(3,2,2,1) + N(2,2)$
for each $N \ge 2$.
\end{example}
 
 



\begin{lemma}\label{lemma:injective}
The map $v \mapsto \bt{v}$ 
defines a $\C$-linear injection
\[ \nablanumuE_\lambda \rightarrow \nablanumuEh_{\lambda+n(1^r)}. \]
\end{lemma}

\begin{proof}
Let $T \in \SSYT_\B(\nu)$. 
Since  $\wt(\bt{T}) = \wt(T) + n(1^r)$, Proposition~\ref{prop:nablaT}(i) 
implies that $F(\bt{T}) \in \nablanumuEh_{\lambda+n(1^r)}$. 
By definition
\begin{equation}\label{eq:FhatT} F(\bt{T}) = \sum_{\tau \in \CPP(\nu)} \bigotimes_{I=1}^{\ell(\nu)}
\prod_{J=1}^{\nu_I} F(\bt{T}_{\pairpq\tau^{-1}}) \sgn(\tau).
\end{equation}
The row-standardization $\overline{t}$ of a tableau $t$ was defined in Definition~\ref{defn:rowstraightening}.
Applying this operation to each entry of $\bt{T}$, we 
define $S(\bt{T}) \in \SSYT_{\B^+}(\nu)$ by 
$S(\bt{T})_{\pairpq} = \overline{\bt{T}_{\pairpq}}$.  
By Proposition~\ref{prop:boundGL}, 
\[ F\bigl( \bt{T}_{\pairpq\tau^{-1}} \bigr) = F\bigl( S(\bt{T})_{\pairpq\tau^{-1}} \bigr) + u_{\pairpq\tau^{-1}} \]
where $u_{\pairpq\tau^{-1}} \in \nabla^{\mu +(1^r)}(E)$ is a linear combination 
of $\GL$-polytabloids $F(s)$ for $s \in \B^+$
such that $\overline{\bt{T}_{\pairpq\tau^{-1}}} \rhd s$. 
Define $V \in \nabla^\nu\bigl(\nabla^{\mu+(1^r)}(E)\bigr)$
by 
\[ F(\bt{T}) = F\bigl( S(\bt{T}) \bigr) + W;\] 
by the previous sentence,
 the vector $W$ is a linear combination of basis elements of $\Sym^\nu \bigl( \nabla^{\mu + (1^r)}(E) \bigr)$
 each of the form 
 \[ \bigotimes_{I=1}^{\ell(\nu)} \prod_{J=1}^{\nu_I} F(u_{\pairpq}), \]
 where the tableaux $u_{\pairpq} \in \B^+$ can be relabelled by a permutation $\tau$
so that $S(\bt{T})_{\pairpq} \unrhd u_{\pairpq\tau}$ for each $\pairpq \in [\nu]$, with at
least one of these dominance relations strict. It follows that
the coefficient of $f\bigl( S(\bt{T}) \bigr)$ in $F(\bt{T})$ comes entirely from $F\bigl( S(\bt{T}) \bigr)$. 
By Lemma~\ref{lemma:dom}, this coefficient is $1$.

Each $\mu+(1^r)$-tableau entry of $S(\bt{T})$ is of the form $\overline{\hskip0.5pt\bt{t}\,}$ where $t \in \B$.
Given $s = \overline{\hskip0.5pt\bt{t}\,}$ one may reconstruct $t$  as follows:
choose, for each $i \in \{1,\ldots, r\}$, a box $(i,j_i)$ containing
$i$ in row $i$ of $s$; now erase the entry in this box, and move each
entry to the right of the now empty box one place to the left; finally delete
the box at the end of row $i$. More formally,
\[ t_{(i,h)} = \begin{cases} s_{(i,h+1)} & \text{if $i \in \{1,\ldots, r\}$ and $h \ge j_i$} \\
s_{(i,h)} & \text{otherwise.}\end{cases} \]
Therefore the map $T \mapsto S(\bt{T})$ is injective.

Let
\[ v = \sum_{T \in \SSYT_\B(\nu)} \alpha_T F(\bt{T})\] 
where not every coefficient is zero.
Choose $T$ so that $S(\bt{T})$ is a maximal element of $\{S(\bt{T}) : \alpha_T \not= 0\}$
in the dominance order. By the previous two paragraphs, the coefficient of $f\bigl(S(\bt{T}))$ 
in $v$ is $\alpha_T$. Hence the map $F(T) \mapsto F(\bt{T})$ is injective.
\end{proof}


\begin{proposition}\label{prop:Brioninjection}
The map $v \mapsto \bt{v}$ restricts to an injective $\C$-linear map from the highest-weight 
vectors in $\nablanumuE$ of weight $\lambda$ to the
highest-weight
vectors in $\nablanumuEh$ of weight $\lambda + (n^r)$.
Moreover if every $U \in \SSYT_\BB(\nu)$ of weight $\lambda + (n^r)$ 
is $r$-saturated 
then the map is bijective.
\end{proposition}

\begin{proof}
The first part follows by combining Proposition~\ref{prop:nablaT}(ii), Lemma~\ref{lemma:technicalhat}
and Lemma~\ref{lemma:injective},
in the same way as Proposition~\ref{prop:BCV-Nbijection}. 
If the final hypothesis holds then every $U \in \SSYT_\BB(\nu)$ is of the form
$\bt{T}$ for some $T \in \SSYT_\B(\nu)$ and so, 
by Lemma~\ref{lemma:numubasis} and Proposition~\ref{prop:nablaT}(i), the map $F(T) \mapsto F(T^\star)$
defines a linear isomorphism $\nabla^\nu\bigl(\nabla^\mu(E) \bigr)_{\lambda} \rightarrow
\nabla^\nu \bigl( \nabla^{\mu+(1^r)}(E) \bigr)_{\lambda + (n^r)}$. 
%
Therefore in this case the restricted map is bijective on highest-weight vectors.
\end{proof}

We need the following sufficient condition for $r$-saturation. 

\begin{lemma}\label{lemma:saturated}
Let $\BBN{M} = \SSYT_{\{1,\ldots, d\}}\bigl( \mu + M(1^r) \bigr)$.
Every element of $\SSYT_\BBN{M}(\nu)$ of weight $\lambda$ is $r$-saturated if
\[ M > n(\mu_1 + \cdots + \mu_{r-1}) + (n-1)\mu_r + \mu_{r+1} -
(\lambda_1 + \cdots + \lambda_r).\] 
\end{lemma}

\begin{proof}
Let $u$ be a $\mu + M(1^r)$-tableau entry of $U \in \SSYT_\BBN{M}(\nu)$. The entries of~$u$ in $\{1,\ldots, r\}$ lie
in its first $r$ rows. Therefore $u$ has at most $\mu_1 + \cdots + \mu_r + Mr$ such entries.
If $U$ is not saturated then it has a $\mu + M(1^r)$-tableau entry $t$ such that $t_{(r,e)} > r$. This $t$
has at most $\mu_1 + \cdots + \mu_{r-1} + M(r-1) + (e-1)$ entries in $\{1,\ldots, r\}$. Since $e = \mu_{r+1}+1$,
this shows that $U$ has at most $(n-1)(\mu_1 + \cdots + \mu_r + Mr) + (\mu_1 + \cdots + \mu_{r-1} + M(r-1) + \mu_{r+1})$
entries in $\{1,\ldots, r\}$. The number of such entries is $\lambda_1 + \cdots + \lambda_r + Mnr$. Therefore
\[ n(\mu_1 + \cdots + \mu_{r-1} + Mr) + (n-1)\mu_r + \mu_{r+1} - M \ge \lambda_1 + \cdots + \lambda_r + Mnr.\]
Rearranging, this implies the lemma.
\end{proof}

\begin{proof}[Proof of Theorem~\ref{thm:Brion}]
By Proposition~\ref{prop:Brioninjection}
and Proposition~\ref{prop:hw} we have
\[ [\nabla^\nu \bigl( \nabla^\mu (E) \bigr) : \nabla^\lambda (E)] \le [\nabla^\nu \bigl( \nabla^{\mu + (1^r)} (E) \bigr) :
\nabla^{\lambda + n(1^r)}(E)].\] 
The first part of
Theorem~\ref{thm:Brion} now follows from~\eqref{eq:schurnabla}.
Now suppose that $N \in \N_0$ and $N \ge n(\mu_1 + \cdots + \mu_{r-1}) + (n-1)\mu_r + \mu_{r+1} -
(\lambda_1 + \cdots + \lambda_r)$. 
By Lemma~\ref{lemma:saturated}, taking $M = N+1$, every element of $\SSYT_{\BBN{(N+1)}}(\nu)$ is $r$-saturated.
Therefore, by Proposition~\ref{prop:Brioninjection}, the map $v \mapsto v^\star$ from
highest-weight vectors in $\nabla^\nu \bigl( \nabla^{\mu + N(1^r)} (E) \bigr)$ to
highest-weight vectors in $\nabla^\nu \bigl( \nabla^{\mu + (N+1)(1^r)} (E) \bigr)$ is a bijection.
The stability result now follows from Proposition~\ref{prop:hw}.
\end{proof}

\begin{example}\label{ex:stability}
Example~\ref{ex:Brion} shows that the stability bound in Theorem~\ref{thm:Brion} may be sharp.
We give an example of the opposite case.
Fix $n \in \N$. It is known (see for example \cite[\S 8.5]{PagetWildonTwisted})
that $\bigwedge^n ( \Sym^2\! E )$ is multiplicity-free. Moreover, the
partitions~$\lambda$ such that $[\bigwedge^n (\Sym^2\! E) : \nabla^\lambda(E)] = 1$ are all 
incomparable under the dominance order, and correspond, by Theorem~\ref{thm:maxls}, to the maximal weights
of the plethystic semistandard tableaux of shape $(2)^{(1^n)}$.
For example, $\bigwedge^3 (\Sym^2 E) = \nabla^{(4,1,1)}(E)
\,\oplus \, \nabla^{(3,3)}(E)$, corresponding to the plethystic tableau whose single column has $(2)$-tableau entries
\[ \Bigl\{ \, \young(11)\, , \young(12)\, , \young(13)\, \Bigr\}, \
\Bigl\{ \, \young(11)\, , \young(12)\, , \young(22)\, \Bigr\}, \]
respectively. More generally, for each $\rr \in \N$, provided that $\dim E \ge \rr$,
$\bigwedge^{\binom{\rr+1}{2}} (\Sym^2\! E)$ has $\nabla^{((\rr+1)^\rr)}(E)$ as an irreducible constituent, 
corresponding to the plethystic semistandard tableau of shape $(2)^{(1^n)}$ 
where $n = \binom{\rr+1}{2}$, 
defined using all $2$-multisubsets of $\{1,\ldots, \rr\}$.

Let $\lambda$ be a partition of $2n$ such that $[\bigwedge^n (\Sym^2\! E) : \nabla^\lambda(E)] = 1$. 
Let 
\[ \bigl\{ \raisebox{-5.5pt}{\thindoubletableau{0.45cm}{0.12cm}{$ a_1$}{$ b_1$}}\, ,
\ldots, \raisebox{-5.5pt}{\thindoubletableau{0.45cm}{0.12cm}{$a_n$}{$ b_n$}} \bigr\} \]
be the entries in the corresponding plethystic semistandard tableau of shape $(2)^{(1^n)}$. Let $N \in \N_0$.
Then the unique plethystic semistandard tableau of shape $(2+N)^{(1^n)}$ and weight $\lambda + (nN)$ 
has $(2+N)$-tableau entries $\{ u_1, \ldots, u_n \}$ where for each~$i$,
\[ u_i = \raisebox{-6pt}{\begin{tikzpicture}[x=0.5cm,y=-0.5cm,line width=0.65pt]\draw(0,0)--(6,0)--(6,1)--(1,1)--(0,1)--(0,0);\draw(1,0)--(1,1);\draw(3,0)--(3,1);\draw(4,0)--(4,1);\draw(5,0)--(5,1);\node at (0.5,0.5) {$1$};\node at (2,0.5) {$\ldots$};\node at (3.5,0.5){$1$};\node at (4.5,0.5) {$a_i$};\node at (5.5,0.5) {$b_i$};\end{tikzpicture}}\, . \]
Hence, by Theorem~\ref{thm:maxls}, $[\bigwedge^n (\Sym^{2+N}E ) : \nabla^{\lambda + (nN)}(E)] = 1$
for all $N \in \N_0$. This stability follows from Theorem~\ref{thm:Brion} for $N \ge 2(n-1) - \lambda_1$.
In the case $[\bigwedge^{\binom{\rr+1}{2}} \Sym^{2+N} (E) : \nabla^{((\rr+1)^\rr)+ (nN)}(E)] = 1$
this bound becomes $N \ge 2(\binom{\rr+1}{2} - 1) - (\rr+1) = \rr^2 - 3$; clearly this can be arbitrarily large.
\end{example}

We end with a combinatorial upper bound for the stable multiplicity. Example~\ref{ex:stability}
shows that the bound is sharp in infinitely many cases.

\begin{proposition}\label{prop:stable}
Let $L$ be the greater 
of $n(\mu_1 + \cdots + \mu_{r-1}) + (n-1)\mu_r + \mu_{r+1} - (\lambda_1 + \cdots + \lambda_r)$
and $0$. Then
 \[ [\nabla^\nu \bigl( \nabla^{\mu+N(1^r)} (E) \bigr) : \nabla^{\lambda + N(n^r)}(E)] \le 
  \bigl| \{ T \in \SSYT_{\BBN{L}}(\nu) : \wt(T) = \lambda + L(n^r) \} \bigr| \]
for all $N \in \N_0$ with $N \ge L$.
\end{proposition}

\begin{proof}
The bound holds when $N = L$ since the right-hand side is
\[ \dim \nabla^\nu \bigl( \nabla^{\mu + L(1^r)} (E) \bigr)_{\lambda + L(r^n)} \]
and by Proposition~\ref{prop:hw} this is an upper bound for the left-hand side.
By Theorem~\ref{thm:Brion} the bound holds for all $N \ge L$.
\end{proof}

\section{Proof of Theorem~\ref{thm:Ikenmeyer}}\label{sec:Ikenmeyer}

It is equivalent to show that 
if $n^\star \in \N_0$, $\lambda^\star \in \Par(mn^\star)$ and
$\Sym^{n^\star} \bigl( \nabla^\mu (E) \bigr)$ has $\nabla^{\lambda^\star}(E)$ as an irreducible constituent
then 
\[ [ \Sym^{n+n^\star} \bigl( \nabla^\mu (E) \bigr) : \nabla^{\lambda + \lambda^\star}(E) ] \ge 
[ \Sym^{n} \bigl( \nabla^\mu (E) \bigr) : \nabla^{\lambda}(E)]. \]

Let $c = [ \Sym^{n} \bigl( \nabla^\mu (E) \bigr) : \nabla^{\lambda}(E)]$.
Choose linearly independent highest-weight vectors $v_1, \ldots, v_c \in \Sym^n \bigl( \nabla^\mu (E) \bigr)$
each of weight $\lambda$. By the hypothesis and Proposition~\ref{prop:hw}, there is a highest-weight vector 
$w \in \Sym^{n^\star} \bigl( \nabla^\mu (E) \bigr)$
of weight $\lambda^\star$. 
Multiplying highest-weight vectors 
in the polynomial algebra $\bigoplus_{r=0}^\infty \Sym^r \bigl(\nabla^\lambda(E)\bigr)$, 
we see that $v_1w, \ldots, v_cw$ are $c$ linearly independent highest-weight vectors each in 
$\Sym^{n+n^\star}\! \bigl( \nabla^\mu (E) \bigr)$ and each of weight $\lambda + \lambda^\star$. The theorem follows.

\section{Proof of Theorem~\ref{thm:maxls}}\label{sec:maxls}

Let $d = \ell(\lambda)$ and let  $E = \langle e_1, \ldots, e_d \rangle$ is a $d$-dimensional
complex vector space.
Let $\B = \SSYT_{\{1,\ldots, d\}}(\mu)$. 
By Definition~\ref{defn:pssyt} the plethystic semistandard tableaux of shape $\mu^\nu$
whose $\mu$-tableau entries have entries from $\{1,\ldots, d\}$ are precisely
the elements of $\SSYT_\B(\nu)$.
By Lemma~\ref{lemma:numubasis}, $\nabla^\nu \bigl( \nabla^\mu (E) \bigr)$
has $\{F(S) : S \in \SSYT_\B(\nu) \}$
as a canonical basis. By Proposition~\ref{prop:nablaT}, if $S \in \SSYT_\B(\nu)$ has weight
$\lambda$ then $F(S) \in \nabla^\nu \bigl(\nabla^\mu(E) \bigr)$ is a weight vector of weight $\lambda$.


We use this canonical basis to prove the following
two results; the second is illustrated in Example~\ref{ex:notMaximal} below.


\begin{lemma}\label{lemma:constituentImpliesTableau}
If $[ \nabla^\nu \bigl( \nabla^\mu (E) \bigr) : \nabla^\lambda(E) ] \ge 1$ then
there exists a plethystic semistandard tableau $T \in \SSYT_\B(\nu)$ such that $\wt(T) = \lambda$.
\end{lemma}

\begin{proof}
Let $v \in \nabla^\nu \bigl( \nabla^\mu (E) \bigr)$ be a highest-weight vector of weight $\lambda$.
Let $v = \sum_{S \in \SSYT_\B(\nu)} c_S F(S)$ be the expression of $v$ in the canonical basis
given by Lemma~\ref{lemma:numubasis}
of $\nabla^\nu \bigl( \nabla^\mu (E) \bigr)$. By Proposition~\ref{prop:nablaT}(i), each $S$ such
that $c_S\not= 0$ has weight~$\lambda$. Take~$T$ to be any such $S$.
\end{proof}

\begin{proposition}\label{prop:maximalTableauImpliesHW}
Suppose that $\lambda$ is maximal in the dominance order on partitions such that there
exists a plethystic semistandard tableau $T \in \SSYT_\B(\nu)$ of weight $\lambda$.
Then $F(T) \in \nabla^\nu \bigl( \nabla^\mu (E) \bigr)$ is a highest-weight vector of weight~$\lambda$.
\end{proposition}

\begin{proof}
By Proposition~\ref{prop:nablaT}(i), $F(T)$ is a weight vector of weight $\lambda$.
Suppose, for a contradiction, that $F(T)$ is not highest-weight. Then there exists $c \in \{2,\ldots, d\}$
such that $X^{(c)} \cdot F(T) \not= 0$. By Proposition~\ref{prop:nablaT}(ii),
$X^{(c)} \cdot F(T)= \sum F(U)$, where each $U$ is obtained from $T$ by changing a single
$c$ to $c-1$ in a $\mu$-tableau entry of~$T$. Thus
each $U$ has weight $\lambda^\star$ where
\[ \lambda^\star_b = \begin{cases} 
\lambda_b+1 & \text{if $b = c-1$} \\
\lambda_b-1 & \text{if $b = c$} \\
\lambda_b & \text{otherwise.} \end{cases} \]
Let
\[ X^{(c)} \cdot F(T) = \sum_{S \in \SSYT_\B(\nu)} c_S F(S) \]
be the expression of $\sum F(U)$ in 
the canonical basis 
of $\nabla^\nu \bigl( \nabla^\mu (E) \bigr)$. 
Choose $S$ such that $c_S \not=0$. Then $\wt(S) = \lambda^\star \rhd \lambda = \wt(T)$.
This contradicts the maximality of~$\lambda$.
\end{proof}

We are now ready to prove Theorem~\ref{thm:maxls}.

\begin{proof}[Proof of Theorem~\ref{thm:maxls}]
By Lemma~\ref{lemma:constituentImpliesTableau}, if 
$[\nabla^\mu \bigl( \nabla^\mu (E) \bigr) : \nabla^{\lambda}(E)]
 \ge 1$ then there is a plethystic semistandard tableau $T\in \SSYT_\B(\nu)$ 
 of weight $\lambda$. Conversely, by Proposition~\ref{prop:maximalTableauImpliesHW}, if $\lambda$ is maximal in the dominance order such that there is a plethystic semistandard tableau $T \in \SSYT_\B(\nu)$ then $F(T)$ is a highest-weight vector, and so
 $[\nabla^\mu \bigl( \nabla^\mu (E) \bigr) : \nabla^\lambda(E) ] \ge 1$. Therefore the 
maximal partitions $\lambda$ 
in the dominance order 
such that $[\nabla^\nu \bigl( \nabla^\mu (E) \bigr) : \nabla^\lambda(E) ] \ge 1$
are precisely the maximal weights of the elements of $\SSYT_\B(\nu)$. This proves the first part of the theorem. 
Now suppose that~$\lambda$ is maximal in the dominance order such that $\nabla^\lambda(E)$ appears in
$\nabla^\nu \bigl( \nabla^\mu (E) \bigr)$. Let $S_{(1)}, \ldots, S_{(r)}$ be the 
plethystic semistandard tableaux of shape $\mu^\nu$ and weight~$\lambda$. 
By the canonical basis in Lemma~\ref{lemma:numubasis} and Proposition~\ref{prop:nablaT}(i),
$F(S_{(1)}), \ldots, F(S_{(r)})$ form a basis for the
weight space
$\nabla^\nu \bigl( \nabla^\mu (E) \bigr)_{\!\lambda}$. By 
Proposition~\ref{prop:maximalTableauImpliesHW} 
 these vectors are highest-weight.
Therefore 
\[ r = [ \nabla^\nu \bigl( \nabla^\mu (E) \bigr) : \nabla^\lambda(E) ] \]
as required.
\end{proof}

As a corollary we obtain the result mentioned in the introduction that is surprisingly non-trivial
to prove entirely combinatorially.

\begin{corollary}\label{cor:partitionWeight}
If $T$ is a plethystic semistandard tableau of maximal weight then $\wt(T)$ is a partition.
\end{corollary}

\begin{proof}
Since the weight of a highest-weight vector is a partition, this is immediate from Proposition~\ref{prop:maximalTableauImpliesHW}.
\end{proof}

By Proposition~\ref{prop:closed}, if $T$ is a plethystic semistandard
tableau of shape $\mu^{(1^n)}$ and maximal weight
then the set of $\mu$-tableau entries of $T$ is closed, in the sense of Definition~\ref{defn:closureAlt}. 
The converse does not hold: we show this in the following example, which makes constructive the 
proof of Proposition~\ref{prop:maximalTableauImpliesHW}.

\begin{example}\label{ex:notMaximal}
For ease of notation we shall identify $\nabla^{(1^n)}\bigl(\nabla^{\mu}(E)\bigr)$ with $\bigwedge^n \bigl(\nabla^\mu(E)
\bigr)$ via the map sending $F(T)$, where $T$ is a plethystic semistandard tableau with $\mu$-tableau
entries $t_1,\ldots, t_n$ read from top to bottom, to $F(t_1) \wedge \cdots \wedge F(t_n)$.

Let $T$ be the plethystic semistandard tableau of shape $(2,2)^{(1^{11})}$ and
weight $(17,11,8,8)$ 
whose $(2,2)$-tableau entries, read from top to bottom are
\[ 
\scalebox{1}{$\displaystyle
\young(11,22)\,,\,\young(11,23)\,,\,\young(12,23)\,,\,\young(11,33)\,,\,\young(12,33)\,,\,\young(11,24)\,,\,\young(12,24)\,,
\young(11,34)\,,\,\young(12,34)\,,\,\young(11,44)\,,\,\young(12,44)$}\,. 
\]
Observe that the set $\mathcal{T}$ of these $(2,2)$-tableaux
is closed. By Proposition~\ref{prop:nablaT}(ii),
$X^{(4)}\cdot F(T)$ has eight summands, 
each obtained by changing an entry of $4$ in the final six $(2,2)$-tableaux above
to $3$. In all but two cases, the new $(2,2$)-tableau obtained is semistandard, and so present in the closed set $\mathcal{T}$;
under our agreed identification, the summand is of the form $ \cdots \wedge\, F(s) \wedge \cdots \wedge\, F(s) \wedge
\cdots $, and so vanishes. Let $u$ and~$u'$ denote the final two $(2,2)$-tableau shown above.
The corresponding summands of $X^{(4)} \cdot F(T)$ are $F(U)$ and $F(U')$ where $U$ and $U'$
are the plethystic semistandard tableaux of shape $(2,2)^{(1^{11})}$ with sets of entries
\[ \mathcal{T} \;\backslash\, \{ u \} \cup \left\{\, \young(11,43) \, \right\} \quad\text{and}\quad
 \mathcal{T} \;\backslash\, \{u'\} \cup \left\{\, \young(12,43) \, \right \}, \]
respectively. By the snake relation defined in \eqref{eq:snake} with $A = \{(2,1)\}$ and $B = \{(1,2),(2,2)\}$
we have, working in $\nabla^{(2,2)}(E)$,
\begin{align*} 
F\left(\, \young(11,43)\, \right) &= F\left(\, \young(11,34) \, \right) + F\left(\, \young(14,13) \,
\right) = F\left(\, \young(11,34) \, \right) \\
F\left(\, \young(12,43)\, \right) &= F\left(\, \young(12,34) \, \right) + F\left(\, \young(14,23) \,
\right) = F\left(\, \young(12,34) \, \right) - F\left(\, \young(13,24) \, \right). 
\end{align*}
Therefore $F(U) = 0$ and $F(U') = - F(T')$ where $T'$ is the 
 plethystic semistandard tableau of shape $(2,2)^{(1^{11})}$ whose entries are the same as $T$,
 except for the final entry $u'$, which is replaced with $\young(13,24)\rule[-12pt]{0pt}{24pt}$\,; since this
 new $(2,2)$-tableau is greater in the total order
 than all the tableaux in $\mathcal{T}$, no reordering within the column is necessary
 in order to make $T'$ semistandard.
Therefore $X^{(4)} \cdot F(T) \not=0$, and so $F(T)$ 
\emph{is not} 
a highest-weight vector. As expected from the proof of Proposition~\ref{prop:maximalTableauImpliesHW}, 
we have obtained a plethystic semistandard tableau  $T'$ of more dominant weight, namely, $(17,11,9,7)$, 
by expressing $X^{(4)} \cdot F(T)$
in the canonical basis of $\nabla^{(1^{11})}\bigl( \nabla^{(2,2)}(E) \bigr)$. 
We leave it to the reader to show that $F(T')$ is a highest-weight vector, and correspondingly,
$T'$ has maximal weight for its shape.
\end{example}

Our final example shows that the converse of Proposition~\ref{prop:maximalTableauImpliesHW} is false.

\begin{example} 
Let $T$ be the plethystic semistandard tableau of shape $(2,1)^{(1^4)}$ with entries
\[ \young(11,2)\, , \ \young(12,2)\, , \ \young(13,2)\,, \ \young(14,2) \]
read from top to bottom. Then $F(T) \in \bigwedge^4 \bigl( \nabla^{(2,1)}(E) \bigr)$ is a highest-weight
vector of weight $(5,5,1,1)$
and so $\langle s_{(1^4)} \circ s_{(2,1)}, s_{(5,5,1,1)}\rangle \ge 1$. However $T$ is not of maximal weight for its shape since $(6,4,2) \unrhd (5,5,1,1)$
and the plethystic semistandard tableau $U$ of shape $(2,1)^{(1^4)}$ with entries
\[ \young(11,2)\, , \ \young(11,3)\, , \ \young(12,2)\,, \ \young(12,3) \]
read from top to bottom has weight $(6,4,2)$. It is easily seen that $U$
has maximal weight in the dominance order, and so $F(U)$ is a highest-weight vector.
In fact there are two plethystic semistandard tableau of shape $(2,1)^{(1^4)}$ and weight $(6,4,2)$,
the second is obtained from $U$ 
by swapping 
the $2$ and $3$ in the final $(2,1)$-tableau entry above. Thus, by Theorem~\ref{thm:maxls},
 $\langle s_{(1^4)} \circ s_{(2,1)},
s_{(6,4,2)} \rangle = 2$; this is one of the smallest examples where the multiplicity of a maximal
constituent is more than $1$.
\end{example}

\section*{Acknowledgements}

We thank the anonymous referee for a very careful reading of an earlier version of this paper
and helpful comments and corrections.

\def\cprime{$'$} \def\Dbar{\leavevmode\lower.6ex\hbox to 0pt{\hskip-.23ex
  \accent"16\hss}D} \def\cprime{$'$}
\providecommand{\bysame}{\leavevmode\hbox to3em{\hrulefill}\thinspace}
\providecommand{\MR}{\relax\ifhmode\unskip\space\fi MR }
\providecommand{\MRhref}[2]{%
  \href{http://www.ams.org/mathscinet-getitem?mr=#1}{#2}
}
\providecommand{\href}[2]{#2}
\renewcommand{\MR}[1]{\relax}

\bibliographystyle{amsplain}

\begin{thebibliography}{10}

\bibitem{Abeasis}
Silvana Abeasis, \emph{The {${\rm GL}(V)$}-invariant ideals in {$S(S^{2}V)$}},
  Rend. Mat. (6) \textbf{13} (1980), no.~2, 235--262. \MR{602662}


\bibitem{BenderKnuth}
Edward~A. Bender and Donald~E. Knuth, \emph{Enumeration of plane partitions},
  J. Combinatorial Theory Ser. A \textbf{13} (1972), 40--54. \MR{0299574 (45
  \#8622)}

\bibitem{Boffi}
Giandomenico Boffi, \emph{On some plethysms}, Adv. Math. \textbf{89} (1991),
  no.~2, 107--126. \MR{1128609}

\bibitem{Brion}
Michel Brion, \emph{Stable properties of plethysm: on two conjectures of
  {F}oulkes}, Manuscripta Math. \textbf{80} (1993), no.~4, 347--371.
  \MR{MR1243152 (95c:20056)}


\bibitem{BCV}
Winfried Bruns, Aldo Conca, and Matteo Varbaro, \emph{Relations between the
  minors of a generic matrix}, Adv. Math. \textbf{244} (2013), 171--206.
  \MR{3077870}

\bibitem{CarreLeclerc}
C.~Carr{\'e} and B.~Leclerc, \emph{Splitting the square of a schur function
  into its symmetric and antisymmetric parts}, J. Alg. Comb. \textbf{4} (1995),
  201--231.

\bibitem{ChoiKimKoWon}
Eun~J. Choi, Young~H. Kim, Hyoung~J. Ko, and Seoung~J. Won, \emph{{${\rm
  GL}_n$}-decomposition of the {S}chur complex {$S_r(\bigwedge^2\phi)$}}, Bull.
  Korean Math. Soc. \textbf{40} (2003), no.~1, 29--51. \MR{1958222}

\bibitem{deBoeck}
Melanie de~Boeck, \emph{A study of {F}oulkes modules using semistandard
  homomorphisms}, arXiv:1409.0734 (2014), 10 pages.

\bibitem{DentThesis}
Suzie Dent, \emph{Incidence structure of partitions}, Ph.D. thesis, UEA, 1997.

\bibitem{DlabRingel}
Vlastimil Dlab and Claus~Michael Ringel, \emph{The module theoretical approach
  to quasi-hereditary algebras}, Representations of algebras and related topics
  ({K}yoto, 1990), London Math. Soc. Lecture Note Ser., vol. 168, Cambridge
  Univ. Press, Cambridge, 1992, pp.~200--224. \MR{1211481}

\bibitem{ErdmannWildon}
Karin Erdmann and Mark~J. Wildon, \emph{Introduction to {L}ie algebras},
  Springer Undergraduate Mathematics Series, Springer-Verlag London, Ltd.,
  London, 2006. \MR{2218355 (2007e:17005)}

\bibitem{Foulkes}
H.~O. Foulkes, \emph{Concomitants of the quintic and sextic up to degree four
  in the coefficients of the ground form}, J. London Math. Soc. \textbf{25}
  (1950), 205--209. \MR{MR0037276 (12,236e)}



\bibitem{FultonHarrisReps}
William Fulton and Joe Harris, \emph{Representation theory, a first course},
  Graduate Texts in Mathematics, vol. 129, Springer, 
  1991.\MR{MR1153249 (93a:20069)}

\bibitem{ErdmannGreenSchockerGLn}
J.~A. Green, \emph{Polynomial representations of $gl_n$, with an appendix on
  {S}chensted correspondence and {L}ittelman paths by K. Erdmann and M.
  Schocker}, 2nd ed., Lecture Notes in Mathematics, vol. 830, Springer, 2007.

\bibitem{HemmerNakano}
David~J. Hemmer and Daniel~K. Nakano, \emph{Specht filtrations for {H}ecke
  algebras of type {A}}, J. London Math. Soc. (2) \textbf{69} (2004), no.~3,
  623--638. \MR{2050037 (2005f:20025)}

\bibitem{Humphreys}
James~E. Humphreys, \emph{Introduction to {L}ie algebras and representation
  theory}, Graduate Texts in Mathematics, vol.~9, Springer-Verlag, New
  York-Berlin, 1978, Second printing, revised. \MR{499562 (81b:17007)}

\bibitem{IkenmeyerThesis}
Christian Ikenmeyer, \emph{Geometric complexity theory, tensor rank, and
  Littlewood--Richardson coefficients}, Ph.D. thesis, Universit{\"a}t
  Paderborn, 2012.


\bibitem{JK}
G.~D. James and A.~Kerber, \emph{The representation theory of the symmetric
  group}, Encyclopedia of Mathematics and its Applications, vol.~16,
  Addison-Wesley Publishing Co., Reading, Mass., 1981. \MR{MR644144
  (83k:20003)}

\bibitem{James}
G.~D. James, \emph{The representation theory of the symmetric groups}, Lecture
  Notes in Mathematics, vol. 682, Springer, Berlin, 1978. \MR{MR513828
  (80g:20019)}

\bibitem{KahleMichalek}
Thomas Kahle and Mateusz Micha{\l}ek, \emph{Plethysm and lattice point
  counting}, Found. Comput. Math. \textbf{16} (2016), no.~5, 1241--1261.
  \MR{3552845}


\bibitem{LoehrRemmel}
Nicholas~A. Loehr and Jeffrey~B. Remmel, \emph{A computational and
  combinatorial expos\'e of plethystic calculus}, J. Algebraic Combin.
  \textbf{33} (2011), no.~2, 163--198. \MR{2765321 (2012c:05344)}

\bibitem{MacDonald}
I.~G. Macdonald, \emph{Symmetric functions and {H}all polynomials}, second ed.,
  Oxford Mathematical Monographs, The Clarendon Press Oxford University Press,
  New York, 1995, With contributions by A. Zelevinsky, Oxford Science
  Publications. \MR{MR1354144 (96h:05207)}
  
\bibitem{McDowellInverseSchur}

Eoghan McDowell, \emph{The image of the Specht module under the inverse Schur functor in arbitrary 
characteristic}, arXiv:2101.057702v2, 27 pages, January 2021.


\bibitem{Newell}
M.~J. Newell, \emph{A theorem on the plethysms of {$S$}-functions}, Quart J.
  Math. Oxford \textbf{2} (1951), 161--166.

\bibitem{PagetFiltration}
Rowena Paget, \emph{A family of modules with {S}pecht and dual {S}pecht
  filtrations}, J. Algebra \textbf{312} (2007), no.~2, 880--890. \MR{MR2333189
  (2008e:20016)}

\bibitem{PagetWildonTwisted}
Rowena Paget and Mark Wildon, \emph{Minimal and maximal constituents of twisted
  {F}oulkes characters}, J. Lond. Math. Soc. \textbf{93} (2016), 301--318.

\bibitem{PagetWildonGeneralizedFoulkes}
Rowena Paget and Mark Wildon, \emph{Generalized {F}oulkes modules and maximal
  and minimal constituents of plethysms of {S}chur functions}, arXiv:1608.04018
  (August 2016), 44.


\bibitem{StanleyII}
Richard~P. Stanley, \emph{Enumerative combinatorics. {V}ol. 2}, Cambridge
  Studies in Advanced Mathematics, vol.~62, Cambridge University Press,
  Cambridge, 1999, With a foreword by Gian-Carlo Rota and appendix 1 by Sergey
  Fomin. \MR{MR1676282 (2000k:05026)}

\bibitem{StanleyPositivity}
\bysame, \emph{Positivity problems and conjectures in algebraic combinatorics},
  Mathematics: frontiers and perspectives, Amer. Math. Soc., Providence, RI,
  2000, pp.~295--319. \MR{1754784 (2001f:05001)}


\bibitem{WildonBound}
Mark Wildon, \emph{Vertices of {S}pecht modules and blocks of the symmetric
  group}, J. Algebra \textbf{323} (2010), no.~8, 2243--2256. \MR{2596377
  (2011f:20019)}

\end{thebibliography}

\end{document}